\documentclass[12pt,twoside]{amsart}
\usepackage{amssymb,amsmath,amsthm, amscd, enumerate, mathrsfs}
\usepackage{graphicx, hhline}
\usepackage[colorlinks=true,pagebackref,hyperindex]{hyperref}
\usepackage[all]{xy}
\usepackage[backrefs, alphabetic, initials]{amsrefs}
\usepackage{color}  
\usepackage{latexsym}
\usepackage[T1]{fontenc} 
\usepackage{fancyhdr}
\usepackage{enumerate}

\title[very basic slc-trivial fibrations]
{Remarks on very basic slc-trivial fibrations}
\author{Haidong Liu}
\date{2020/10/21, version 0.03}
\address{Peking University, Beijing International Center for Mathematical Research, 
Beijing, 100871, China}
\email{hdliu@bicmr.pku.edu.cn}

\DeclareMathOperator{\Supp}{Supp}

\DeclareMathOperator{\rank}{rank}

\DeclareMathOperator{\mult}{mult}
\DeclareMathOperator{\Pic}{Pic}
\DeclareMathOperator{\Ex}{Ex}
\DeclareMathOperator{\Div}{Div}
\DeclareMathOperator{\Exc}{Exc}
\DeclareMathOperator{\id}{id}
\DeclareMathOperator{\Gr}{Gr}
\DeclareMathOperator{\GL}{GL}
\DeclareMathOperator{\prim}{prim}

\newtheorem{thm}{Theorem}[section]
\newtheorem{lem}[thm]{Lemma}
\newtheorem{prop}[thm]{Proposition}
\newtheorem{conj}[thm]{Conjecture}
\newtheorem{cor}[thm]{Corollary}

\theoremstyle{definition}

\newtheorem{defn}[thm]{Definition}
\newtheorem{rem}[thm]{Remark}
\newtheorem*{ack}{Acknowledgments}  
\newtheorem{say}[thm]{}

\newtheorem{case}{Case}

\makeatletter
    
    \@addtoreset{equation}{section}
\makeatother

\begin{document}

\begin{abstract}
We study very basic slc-trivial fibrations.
We show that restricting on any union of lc centers of a very basic slc-trivial fibration, a $\mathbb Q$-Cartier moduli part is numerically trivial if and only if it is $\mathbb Q$-linearly trivial.
We then prove that the abundance conjecture for very basic slc-trivial fibrations holds true when the base is of dimension 2 and the moduli part is $\mathbb Q$-Cartier.
As an application, we prove that the log canonical ring of a projective 
plt pair with Kodaira dimension 3 is finitely generated. 
\end{abstract}

\maketitle 
\tableofcontents

\section{Introduction}\label{sec1}

We begin with Kodaira's canonical bundle formula 
in a generalized form due to Ueno \cite{ueno}, Kawamata \cite{kawamata-hist} and Fujita \cite{fujita-hist}.
Let $f \colon S\to C$ be an elliptic surface 
and $J \colon C\to \mathbb P^1$ the $j$-function.
Then there are $\mathbb Q$-divisors $B_S$, $B_C$ and $M_C$ such that
\begin{equation}
K_S+B_S=f^*(K_C+B_C+M_C)
\end{equation}
where $(S, B_S)$ and $(C, B_C)$ are log pairs related by adjunction, and 
\begin{equation}\label{eq1.1}
\mathcal O_C(12M_C)\simeq J^*\mathcal O_{\mathbb P^1}(1)
\end{equation}
which implies that $M_C$ is semi-ample.
These results can be generalized further to the case that
$f \colon X\to Y$ is an {\em{elliptic fibration}}, that is, an algebraic fiber space 
whose general fiber is an elliptic curve,
and the $j$-function extends to a morphism $J\colon Y\to \mathbb P^1$.
In this setting,
Kodaira's original canonical bundle formula is a special case 
where $f$ is minimal and $\dim Y=1$.

This generalized formula 
plays a central role in the study of elliptic fibrations.
It is natural to expect that some more general kinds of
canonical bundle formulas could play a similar role in the study of the {\em{K-trivial fibration}},
that is, a projective surjective morphism $f \colon (X, B_X)\to  Y$ where $(X, B_X)$ has mild singularities
and $K_X+B_X$ is $\mathbb Q$-linearly trivial over $Y$.

In this direction, Mori's work \cite{mori} is one of the starting points.
It is a prototype of the so-called Fujino--Mori canonical bundle formula,
established by Fujino and Mori in \cite{fujino-mori} combining Kawamata's work \cite{kawamata-2} and an unpublished preprint of Mori's.
At around the same time, Ambro started to study some applications of \cite{kawamata-2}.
He formulated and studied lc-trivial fibrations in \cite{ambro-shokurov}, which is later renamed as
klt-trivial fibrations by Fujino and Gongyo to distinguish with 
their lc-trivial fibrations formulated in \cite{fujino-gongyo}.
In that paper, Fujino and Gongyo
 further showed how to reduce some problems for lc-trivial fibrations 
(Fujino--Gongyo's lc-trivial fibrations)
to those for klt-trivial fibrations (Ambro's lc-trivial fibrations). 
Recently, Fujino generalized klt-trivial fibrations and lc-trivial fibrations to the so-called 
basic slc-trivial fibrations in \cite{fujino-slc-trivial}, where using some deep results of theory of
variation of mixed Hodge structures on cohomology with compact support.
For more details about the involving history, see \cite{ambro-shokurov}*{Introduction}, \cite{kollar-hist}
and \cite{fujino-slc-trivial}*{1.15} for the surveys.

The following definition summarizes the definitions of klt-trivial fibrations \cite{ambro-shokurov}*{Definition 2.1}, 
lc-trivial fibrations \cite{fujino-gongyo}*{Definition 3.2} and 
basic slc-trivial fibrations \cite{fujino-slc-trivial}*{Definition 4.1} simultaneously. 
They are all slightly different from the original ones.
Moreover, we define a special case of basic slc-trivial fibrations,
which is similar to {\em{qlc pairs}} in \cite{ambro}, \cite{fujino-foundations}.

\begin{defn}\label{slc-def}
A {\em{pre-basic slc-trivial}} (resp. {\em{lc-trivial}},
{\em{klt-trivial}}) {\em{fibration}} 
$f \colon (X, B_X) \to (Y,D)$
consists of a projective surjective morphism $f\colon X\to Y$ and
a simple normal crossing pair $(X, B_X)$ satisfying the following properties:
\begin{itemize}
\item[$(1)$] $Y$ is a normal variety,
\item[$(2)$] every stratum of $X$ is dominant onto $Y$ and $f_*\mathcal O_X \simeq \mathcal O_Y$,
\item[$(3)$] $(X,B_X)$ is sub-slc (resp. sub-lc, sub-klt) over $Y$,
\item[$(4)$] $D$ is a  $\mathbb Q$-Cartier $\mathbb Q$-divisor on $Y$ 
such that $K_X+B_X\sim_{\mathbb Q}f^*D$.
\end{itemize} 
If a pre-basic slc-trivial (resp. lc-trivial, klt-trivial) fibration $f \colon (X, B_X) \to (Y,D)$ satisfies
\begin{itemize}
\item[$(5)$] $\rank f_*\mathcal O_X(\lceil -(B^{<1}_X)\rceil)=1$,
\end{itemize}
then it is called a {\em{basic slc-trivial}} (resp. {\em{lc-trivial, klt-trivial}}) {\em{fibration}};
furthermore, if the basic slc-trivial (resp. lc-trivial, klt-trivial)  
fibration also satisfies
\begin{itemize}
\item[$(6)$] the natural map $\mathcal O_Y\to f_*\mathcal O_X(\lceil -(B^{<1}_X)\rceil)$ is an isomorphism,
\end{itemize}
then it is called a {\em{very basic slc-trivial}} (resp. {\em{lc-trivial, klt-trivial}}) {\em{fibration}}.
\end{defn}

\begin{rem}\label{rem1.2}
It is more natural (cf. \cite{fujino-foundations}*{Remark 6.2.3}) to replace the condition (4) above by that 
\begin{itemize}
\item[$(4')$] $\mathcal D$ is a $\mathbb Q$-line bundle on $Y$ such that $K_X+B_X\sim_{\mathbb Q}f^*\mathcal D$.
\end{itemize} 
Let $\Div(X)$ be the group of Cartier divisors on $X$ and $\Pic(X)$ the Picard group of $X$. Let 
$$
\delta_X\colon \Div(X)\otimes \mathbb Q \to \Pic(X)\otimes \mathbb Q 
$$
be the homomorphism induced by $A \to \mathcal O_{X}(A)$ where $A$ is a Cartier divisor on $X$.
Then
$K_X+B_X\sim_{\mathbb Q}f^*\mathcal D$ means that
$$
\delta_X(K_X+B_X)=f^*\mathcal D
$$
in $ \Pic(X)\otimes \mathbb Q $. 
Let $b=\min \{m\in \mathbb Z_{>0}| m(K_F+B_F)\sim 0\}$
where $F$ is a general fiber of $f$ and $K_F+B_F=(K_X+B_X)|_F$.
It follows that we can take a  rational function $\varphi\in \Gamma (X, \mathcal K^*_X)$
and a $\mathbb Q$-Cartier $\mathbb Q$-divisor $D$ on $Y$
(see \cite{fujino-slc-trivial}*{Section 6})
such that $\mathcal O_Y(D)=\mathcal D$ in $ \Pic(Y)\otimes \mathbb Q$ and
\begin{equation}\label{eq1.2}
K_X+B_X+\frac{1}{b}(\varphi)=f^*D.
\end{equation}
Note that $D$ depends on the choice of $\varphi$.
Then, when we use the condition $(4)$ of Definition \ref{slc-def},
we always implicitly replace $D$ by a $\mathbb Q$-linearly equivalent divisor 
satisfying \eqref{eq1.2}. This replacement happens in 
\cites{ambro-shokurov, fujino-gongyo, fujino-slc-trivial, fl-plt, liu} 
and others without mentioning.
Following the literatures, 
we abuse of terminology throughout this paper
 if there is no risk of confusion.
\end{rem}

\begin{rem}\label{rem1.3}
A {\em{qlc pair}} is a data $(Y,D, f\colon (X,B_X) \to Y)$
where $f$ is a proper morphism from a globally
embedded simple normal crossing pair $(X,B_X)$, $B_X$ is a sub boundary, $D$ is 
a $\mathbb Q$-Cartier $\mathbb Q$-divisor (or a $\mathbb Q$-line bundle) on $Y$ such that
$K_X+B_X\sim_{\mathbb Q}f^*D$, and $\mathcal O_Y\simeq f_*\mathcal O_X(\lceil -(B^{<1}_X)\rceil)$.
 A {\em{normal qlc pair}} is a qlc pair $(Y,D, f\colon (X,B_X) \to Y)$
where $Y$ is normal.
By \cite{fujino-slc-trivial}*{Theorem 10.3},
after some modifications of $(X,B_X)$,
a normal qlc pair is a very basic slc-trivial fibration.
Conversely, the pair $(X,B_X)$ in a very basic slc-trivial fibration is not necessary to be globally
embedded simple normal crossing, thus
a very basic slc-trivial fibration is not necessarily a normal qlc pair.
But in practice, the very basic slc-trivial fibration we meet is always a normal qlc pair,
e.g., the constructions in the proofs of \cite{fl-plt}*{Theorem 1.2} and \cite{liu}*{Theorem 1.1},
or at least, some techniques of the qlc theory still work.
For more details about qlc pairs, see \cites{ambro, fujino-foundations, fl-db} and the  
references therein.
\end{rem}

We introduce the notion of very basic slc-trivial fibrations because they show up in the study of the
finite generation conjecture whose lc pair is not big, as we have seen in \cites{fl-plt, liu}.
The main purpose of this paper is a continuous work of {\em{loc. cit.}} to 
prove the finite generation for plt pairs with Kodaira dimension 3.

\begin{thm}\label{main-cor}
Let $(X, \Delta)$ be a projective plt pair such 
that $\Delta$ is a $\mathbb Q$-divisor. 
Assume that $\kappa (X, K_X+\Delta)=3$. 
Then the log canonical ring 
$$
R(X, \Delta)=\bigoplus _{m\geq 0}H^0(X, \mathcal 
O_X(\lfloor m(K_X+\Delta)\rfloor))
$$ 
is a finitely generated $\mathbb C$-algebra. 
\end{thm}

To this end, we study very basic slc-trivial fibrations in details and prepare some more
general and technical results.
Let $f\colon (X, B_X) \to (Y,D)$ be a pre-basic slc-trivial fibration.
Put 
$$
B_Y:=\sum _P (1-b_P)P, 
$$ 
where $P$ runs over prime divisors on $W$ and 
$$
b_P:=\max \left\{t \in \mathbb Q \mid
 {\text{$(X, B_X+tf^*P)$ is sub-slc over the generic point of $P$}} \right\}.  
$$ 
The $\mathbb Q$-divisor $B_Y$ is well-defined 
and
is called the {\em{discriminant $\mathbb Q$-divisor}} (or {\em{discriminant part}})
of $f\colon (X, B_X) \to (Y,D)$.
We further put 
$$
M_Y:=D-K_Y-B_Y, 
$$
where $K_Y$ is the canonical divisor of $Y$. 
Then the $\mathbb Q$-divisor $M_Y$ is called the {\em{moduli $\mathbb Q$-divisor}} 
(or {\em{moduli part}}) of $f\colon (X, B_X) \to (Y,D)$.
Usually we call $D=K_Y+B_Y+M_Y$ the {\em{structure decomposition}}.
See \cite{fujino-slc-trivial} for more details.

\begin{conj}[\cite{fl-normal}*{Conjecture 1.5}]\label{lc-conj}
Let $f\colon (X, B_X) \to (Y,D)$ be a very basic slc-trivial fibration
with the structure decomposition $D=K_Y+B_Y+M_Y$.
Then there exists an effective $\mathbb Q$-divisor $\Delta\sim_{\mathbb Q}B_Y+M_Y$ on 
$Y$ such that $(Y, \Delta)$ is log canonical.
\end{conj}

It is easy to see that 
Conjecture \ref{lc-conj} holds true  
in the case that  $M_Y$ is $\mathbb Q$-Cartier and numerically trivial 
by Corollary \ref{cart-cor} and \cite{fujino-fujisawa-liu}*{Theorem 1.3}.
In this paper, 
we study the numerically trivial moduli part further and prove the following useful theorem,
which is needed in the proof of Theorem \ref{main-thm}.
Note that the lc center of a very basic slc-trivial fibration is defined in Definition \ref{def-lc}.
Note also that the union of lc centers in our setting is not necessary to be normal.

\begin{thm}[Theorem \ref{key thm}]\label{thm1.6}
Let $f\colon (X, B_X) \to (Y,D)$ be a very basic slc-trivial fibration
with the structure decomposition $D=K_Y+B_Y+M_Y$.
Let $Z$ be a union of lc centers of $f$.
If  $M_Y$ is $\mathbb Q$-Cartier and $M_Y|_Z\equiv 0$, then $M_Y|_Z\sim_{\mathbb Q}0$.
\end{thm}

If we assume that the abundance conjecture for lc pairs holds true, then Conjecture \ref{lc-conj}
implies the following conjecture immediately.

\begin{conj}\label{ab-conj}
Let $f\colon (X, B_X) \to (Y,D)$ be a very basic slc-trivial fibration.
If $D$ is nef, then $D$ is semi-ample.
\end{conj}

Conjecture \ref{ab-conj} can be 
viewed as a kind of abundance conjectures for very basic slc-trivial fibrations.
It holds true 
in the case that $\dim Y=1$ by using \cite{fujino-fujisawa-liu}*{Corollary 1.4} directly
(cf. \cite{fl-plt}*{Corollary 5.4} or \cite{liu}*{Corollary 3.2}).
We show that it also holds true when $\dim Y=2$ and
$M_Y$ is $\mathbb Q$-Cartier, equivalently, $K_Y+B_Y$ is $\mathbb Q$-Cartier. 
More precisely,

\begin{thm}\label{main-thm}
Let $f\colon (X, B_X) \to (Y,D)$ be a very basic slc-trivial fibration
with the structure decomposition $D=K_Y+B_Y+M_Y$, where
 $Y$ is a normal surface. 
If $K_Y+B_Y$ is $\mathbb Q$-Cartier and $D$ is nef, then $D$ is semi-ample.
\end{thm}

The assumption that $K_Y+B_Y$ is $\mathbb Q$-Cartier is very natural 
in practice. Actually, we can even assume that $Y$ is $\mathbb Q$-factorial
if we only use Theorem \ref{main-thm} to prove Theorem \ref{main-cor}.
It is also worth to mention that we only need 
Theorem \ref{main-thm} for very basic klt-trivial fibrations 
in the proof of Theorem \ref{main-cor}.
See Section \ref{sec6} for the details.

\begin{ack}
The author is very grateful to Osamu Fujino, Jingjun Han, Chen Jiang and Zhiyu Tian for useful 
discussions and suggestions. He would like to thank Jingjun Han for his private note on 
the termination of the LMMP of generalized lc pairs in dimension 3.
 He would also like to thank Zhiyu Tian 
for bringing \cite{gglr} to his attention, and Colleen Robles for answering some naive questions
on  \cite{gglr}.
\end{ack}


\section{Preliminaries}\label{sec2}

\begin{say}[\textbf{Notation and conventions}]\label{2.1}
We work over the complex number field $\mathbb C$ throughout 
this paper. We freely use the basic 
notation of the minimal model program in 
\cites{fujino-fundamental, fujino-surfaces, fujino-foundations}. 
In this paper, we consider only $\mathbb Q$-divisors instead of $\mathbb R$-divisors. 
A {\em scheme} is always assumed to be separated and of finite type over $\mathbb{C}$.
 A {\em variety} is a reduced and irreducible scheme. 
A {\em curve} (resp. {\em surface}) is a variety of dimension 1 (resp. 2).

Let $D$ be a $\mathbb Q$-divisor on an equidimensional scheme $X$, that is, 
$D$ is a finite formal sum $\sum _i d_i D_i$ where 
$d_i\in \mathbb{Q}$ and $\{D_i\}_i$ are distinct prime divisors 
on $X$. 
We put 
$$
D^{<1}=\sum _{d_i<1}d_iD_i, \quad 
D^{\leq 1}=\sum _{d_i\leq 1}d_i D_i, \quad 
\text{and} \quad D^{=1}=\sum _{d_i=1}D_i. 
$$
We also put 
$$
\lceil D\rceil =\sum _i \lceil d_i \rceil D_i, \quad 
\lfloor D\rfloor=-\lceil -D\rceil, \quad \text{and} 
\quad 
\{D\}=D-\lfloor D\rfloor, 
$$
where $\lceil d_i\rceil$ is the integer defined by 
$d_i\leq \lceil d_i\rceil <d_i+1$. A $\mathbb Q$-divisor 
$D$ is called a {\em{sub boundary}} if $D=D^{\leq 1}$ holds.
if a sub boundary $D$ is also effective ($d_i\geq 0$ for every $i$), then $D$ is called  a {\em{boundary}}.
Let $D_1$ and $D_2$ be two $\mathbb Q$-divisors on $X$. 
We denote $D_1\leq D_2$ (resp. $D_1< D_2$) if $D_2-D_1$ is effective (resp. strictly effective).
Assume that $f\colon X\to Y$ is a surjective morphism onto a normal variety $Y$.
Then 
$$
D^v=\sum_{f(D_i)\subsetneq Y} d_iD_i \quad 
\text{and} \quad D^h=D-D^v
$$
are called the {{\em{vertical part}} and \em{horizontal part}} of $D$ with respect to $f\colon X\to Y$
respectively.

A $\mathbb Q$-Cartier divisor $D$ on $X$ is an element of 
$\Gamma(X, \mathcal K_X^*/\mathcal O_X^*)\otimes \mathbb Q$.
Let $D_1$ and $D_2$ be two $\mathbb Q$-Cartier divisors 
on $X$. Then 
we write $D_1\sim _{\mathbb Q} D_2$ if there exists a positive integer $m$ 
such that $mD_1\sim mD_2$, that is, 
$mD_1$ is linearly equivalent to $mD_2$. 
Let $f\colon X\to Y$ be a surjective morphism onto a normal variety $Y$.
Then we write $D_1\sim _{\mathbb Q, f}D_2$ if there exists a $\mathbb Q$-Cartier 
divisor $B$ on $Y$ such that $D_1-D_2\sim _{\mathbb Q}f^*B$. 
If $P$ is a prime divisor on $Y$, then its generic point is Cartier.
We denote $f^*P_{\eta}$ the closure of the pullback of the generic point of $P$.
If $P=\sum a_iP_i$ is a $\mathbb Q$-divisor, 
then $f^*P_{\eta}:=\sum a_i f^*P_{i,\eta}$.
\end{say}

\begin{say}[\textbf{Singularities of pairs}]\label{2.2}
A {\em pair} $(X, \Delta)$ consists of a normal 
variety $X$ and a $\mathbb Q$-divisor 
$\Delta$ on $X$ such that $K_X+\Delta$ is $\mathbb Q$-Cartier. 
Let $f\colon Y\to X$ be a projective 
birational morphism from a normal variety $Y$. 
Then we can write 
$$
K_Y=f^*(K_X+\Delta)+\sum _E a(E, X, \Delta)E
$$ 
with 
$$f_*\left(\underset{E}\sum a(E, X, \Delta)E\right)=-\Delta, 
$$ 
where $E$ runs over prime divisors on $Y$. 
We call $a(E, X, \Delta)$ the {\em{discrepancy}} of $E$ with 
respect to $(X, \Delta)$. 
Note that we can define the discrepancy $a(E, X, \Delta)$ for 
any prime divisor $E$ over $X$ by taking a suitable 
resolution of singularities of $X$. 
If $a(E, X, \Delta)\geq -1$ (resp. $ a(E, X, \Delta)>-1$) for 
every prime divisor $E$ over $X$, 
then $(X, \Delta)$ is called {\em{sub log canonical}} ({\em{sub-lc}} for short) 
or {\em{sub Kawamata log terminal}} ({\em{sub-klt}} for short) respectively.
If $a(E, X, \Delta)>-1$ for every exceptional divisor $E$ over 
$X$, then $(X, \Delta)$ is called {\em{sub purely log terminal}} ({\em{sub-plt}} for short).  

If $\Delta$ is effective, then a pair $(X, \Delta)$ is called 
{\em{lc}} (resp. {\em{klt, plt}}) if it is sub-lc (resp. sub-klt, sub-plt). 
A {\em{divisorial log terminal}} ({\em{dlt}} for short)
pair is a limit of klt pairs in the sense of \cite{kollar-mori}*{Proposition 2.43}
(see \cite{kollar-mori}*{Definition 2.37 and Proposition 2.40} for precise definitions).

Let $(X, \Delta)$ be a sub-lc pair. 
If there exists a projective birational morphism 
$f\colon Y\to X$ from a normal variety $Y$ and a prime divisor $E$ on $Y$ 
with $a(E, X, \Delta)=-1$, then $f(E)$ is called an {\em{lc center}} of 
$(X, \Delta)$.

Let $X$ be a reduced equidimensional scheme which satisfies Serre's $S_2$ condition and 
is normal crossing in codimension one. 
Let $\Delta$ be an effective $\mathbb Q$-divisor 
on $X$ such that no irreducible component of $\Supp \Delta$ is 
contained in the singular locus of $X$ and $K_X+\Delta$ is $\mathbb Q$-Cartier. 
We say that $(X, \Delta)$ is a {\em{sub semi log canonical}} 
({\em{sub-slc}} for short) pair if 
$(X^\nu, \Delta_{X^\nu})$ is sub log canonical, 
where $\nu\colon X^\nu\to X$ is the normalization of 
$X$ and $K_{X^\nu}+\Delta_{X^\nu}=\nu^*(K_X+\Delta)$, 
that is, $\Delta_{X^\nu}$ is the sum of the inverse image of $\Delta$ 
and the conductor of $X$. 
We say that  $(X, \Delta)$ is {\em{semi log canonical}}
({\em{slc}} for short) if $\Delta$ is effective.

An {\em{slc center}} of a sub-slc pair $(X, \Delta)$ is the $\nu$-image 
of an lc center of $(X^\nu, \Delta_{X^\nu})$. 
An {\em{slc stratum}} of $(X, \Delta)$ means 
either an slc center of $(X, \Delta)$ or an irreducible component of 
$X$. 
For more details of semi log canonical pairs, see \cites{fujino-slc, kollar92, kollar}. 
\end{say}

\begin{say}[\textbf{Kodaira dimension}]\label{2.3}
Let $D$ be a $\mathbb Q$-Cartier $\mathbb Q$-divisor 
on a normal projective variety $X$. Let $m_0$ be a positive integer 
such that $m_0D$ is a Cartier divisor. 
Let $$\Phi_{|mm_0D|}\colon X\dashrightarrow \mathbb P^{\dim\!|mm_0D|}$$ be 
the rational map given by the complete linear system $|mm_0D|$ 
for a positive integer $m$. 
Note that $\Phi_{|mm_0D|}(X)$ denotes the 
closure of the image of the rational map $\Phi_{|mm_0D|}$. 
We put 
$$\kappa (X, D):=\max_m \dim \Phi_{|mm_0D|}(X)$$ if 
$|mm_0D|\ne \emptyset$ for some $m$ and 
$\kappa (X, D):=-\infty$ otherwise. 
We call $\kappa(X, D)$ ($\kappa(D)$ for short if there is no danger of confusion)
the {\em{Iitaka dimension}} 
of $D$, and  $\kappa(X, K_X+\Delta)$  the {\em{log Kodaira dimension}} of 
$(X,\Delta)$ when $(X,\Delta)$ is a projective lc pair.
We simply call it {\em{Kodaira dimension}} for short in this paper.

We say that $D$ is {\em{big}} if $\kappa(X, D)=\dim X$,
and that a projective lc pair $(X,\Delta)$ is {\em{big}} if $\kappa(X, K_X+\Delta)=\dim X$.
We say that $D$ is {\em{nef}}
if $D\cdot C\geq 0$ for every curve $C$ on $X$.
For a general $\mathbb Q$-divisor $D$,
instead of using {\em{b-divisor}} (cf. \cite{fujino-slc-trivial}*{(2.10)}, 
\cite{corti}*{2.3.2}),
we say that $D$ is nef if there exists a projective birational morphism 
$f\colon Y\to X$ from a normal variety $Y$ and a nef
$\mathbb Q$-Cartier $\mathbb Q$-divisor $B$ on $Y$
such that $D=f_*B$.
\end{say}

\begin{say}[\textbf{Nef dimension}]\label{2.4}
Let $D$ be a nef $\mathbb Q$-Cartier $\mathbb Q$-divisor on a normal projective variety $X$.
By \cite{b8}*{Theorem 2.1} (see also \cite{fujino-surfaces}*{2.6} and \cite{ambro-nef}*{Section 2}), 
there exists an almost holomorphic, dominant rational
map $f\colon X\dasharrow Y$ with connected fibers, called a {\em{reduction map}} associated to $D$ such that
\begin{itemize}
\item[(i)] $D$ is numerically trivial on all compact fibers $F$ of $f$ with $\dim F=\dim X- \dim Y$,
\item[(ii)] if $x\in X$ is a very general point and  $C$ is an irreducible curve through
$x$ with $\dim f(C)>0$, then $D\cdot C>0$.
\end{itemize} 
The map $f$ is unique up to birational equivalence of $Y$. 
We call $n(X,D):=\dim Y$ the {\em{nef dimension}} of $D$
($n(D)$ for short if there is no danger of confusion).
It is well-known that $n(X,D)\geq 0$ and $\kappa(X, D)\leq n(X,D)\leq \dim X$.

In Section \ref{sec5}, we will prove Theorem \ref{main-thm} case by case according to 
the nef dimension $n(D)$ instead of the Iitaka dimension  $\kappa (D)$.
The advantage is to avoid proving the non-vanishing theorem for the moduli part directly.
\end{say}


\section{On very basic slc-trivial fibrations}\label{sec3}
Throughout this section, we fix the notation that 
$f\colon (X, B_X) \to (Y,D)$ is a very basic slc-trivial fibration
with the structure decomposition $D=K_Y+B_Y+M_Y$.
Two sheaves are isomorphic in codimension $k$ means they are
isomorphic outside a subset of codimension $k+1$.

First, we show some
properties about the discriminant part $B_Y$ and the moduli part $M_Y$.
The proof of the following proposition is the same as that of \cite{fl-plt}*{Lemma 5.1}.

\begin{prop}\label{dis-bound}
$B_Y$ is a boundary.
\end{prop}

The effectiveness of $B_Y$ follows from the condition $(6)$ of Definition \ref{slc-def}.
This property is one of the reasons to consider very basic slc-trivial fibrations.
Moreover, it is easy to check that in this case,
the pair $(Y, B_Y+M_Y)$ is a {\em{generalized lc pair}} in the sense of 
\cite{bz}*{Definition 4.1}. Combining Remark \ref{rem1.3}, we see that 
very basic slc-trivial fibrations
sit between normal qlc pairs and generalized lc pairs.

\begin{cor}\label{cart-cor}
If $K_Y+B_Y$ is $\mathbb Q$-Cartier, then $(Y, B_Y)$ is log canonical.
\end{cor}
\begin{proof}
 By assumption and \cite{fujino-slc-trivial}*{Theorem 5.1}, $(Y, B_Y)$ is sub log canonical.
Thus it is log canonical by Proposition \ref{dis-bound}.
\end{proof}

We immediately get the following proposition about the moduli part by \cite{fujino-slc-trivial}*{Theorem 1.3}.

\begin{prop}\label{moduli-nef}
$M_Y$ is nef. That is, there exists a birational morphism $p: Y'\to Y$ and 
a nef $\mathbb Q$-Cartier $\mathbb Q$-divisor $M_{Y'}$ on $Y'$
such that $M_Y=p_*M_{Y'}$.
\end{prop}

\begin{rem}\label{nef-hist}
In Kodaira's canonical bundle formula for elliptic fibrations, 
the moduli part is semi-ample by \eqref{eq1.1}. 
It is natural to wonder the abundance of the moduli part in general cases.
In \cite{kawamata-hist}*{Theorem 2}, 
Kawamata essentially proved that the moduli part of a basic klt-trivial fibration is nef. 
After he learned \cite{kawamata-hist}, Fujino
discussed the semi-ampleness of the moduli part for certain algebraic fiber spaces in \cite{fujino-bundle}
and showed that the moduli part of a basic lc-trivial fibration is nef in \cite{fujino-bundle-a}.
In \cite{ambro-shokurov}, Ambro proved that the moduli part of a basic klt-trivial fibration 
is {\em{b-nef}} and {\em{abundant}} under some assumptions, and is semi-ample when $\dim Y=1$.
Later Fujino--Gongyo 
generalized these results for basic lc-trivial fibrations in \cite{fujino-gongyo}.
Soon after Fujino posted the b-semi-ampleness conjecture for basic slc-trivial fibrations
in his recent work \cite{fujino-slc-trivial},
Fujino--Fujisawa--Liu \cite{fujino-fujisawa-liu} proved that
b-semi-ampleness conjecture holds true in $\dim Y=1$.
It is worth to mention that there are effective versions of the abundance of the moduli part.
In this direction, see \cites{ps, floris} and the references therein. 
\end{rem}

Next, we prepare some general lemmas for very basic slc-trivial fibrations.
The first lemma is quite similar to \cite{fl-db}*{Lemma 2.2}. 

\begin{lem}\label{slc-lem}
Let $f\colon (X, B_X) \to (Y,D)$ be a very basic slc-trivial fibration
with $K_X+B_X+\frac{1}{b}(\varphi)=f^*D$.
Let $g\colon Y\to Y'$ be a proper surjective morphism 
between varieties such that $g_*\mathcal O_Y\simeq \mathcal O_{Y'}$. 
Assume that $D=g^*D'+E$ holds for some $\mathbb Q$-Cartier 
 $\mathbb Q$-divisor $D'$ on $Y'$, and $rE$ is an effective Cartier divisor 
for some integer $r\geq 1$ such that
$\mathcal O_{Y'} \simeq  g_*\mathcal O_Y(rE)$. 
Then $g\circ f\colon (X, B_X-f^*E) \to (Y',D')$ is an induced very basic slc-trivial fibration.
\end{lem}
\begin{proof}
Since $Y$ is normal and $g_*\mathcal O_Y\simeq \mathcal O_{Y'}$,
$Y'$ is also normal.
Consider the induced morphism $g\circ f\colon X \to Y'$ with 
$$
K_X+B_X-f^*E+\frac{1}{b}(\varphi)=f^*g^*D'.
$$
Let $B'_{X}=B_X-f^*E$.
It is easy to check conditions (1)--(4) of Definition \ref{slc-def} one by one.
For the condition (6), which implies the condition (5), 
we have that $B'^{<1}_X+B'^{=1}_X=B^{<1}_X+B^{=1}_X-f^*E$ and $B'^{=1}_X\leq B^{=1}_X$.
Therefore,
$$
\mathcal O_{Y'}\hookrightarrow g_*f_*\mathcal O_X(\lceil -(B'^{<1}_X)\rceil)\hookrightarrow
g_*f_*\mathcal O_X(\lceil -(B^{<1}_X)+f^*E\rceil) \hookrightarrow
g_*\mathcal O_Y(rE) \simeq \mathcal O_{Y'}.
$$
That is, $\mathcal O_{Y'}\hookrightarrow g_*f_*\mathcal O_X(\lceil -(B'^{<1}_X)\rceil)$
is an isomorphism. 
\end{proof}

\begin{rem}\label{push-rem}
Note that  the morphism $g\colon Y\to Y'$ above is not necessary to be birational. 
In this paper, we only use Lemma \ref{slc-lem} 
in the case that either $g$ is birational or $E=0$.
\end{rem}

The second lemma is similar to \cite{liu-fujita}*{Lemma 2.10}.
Let $f\colon (X, B_X) \to (Y,D)$ be a very basic slc-trivial fibration
with the structure decomposition $D=K_Y+B_Y+M_Y$.
Let $p\colon Y'\to Y$ be a birational morphism between normal varieties.
By \cite{fujino-slc-trivial}*{Section 4, Theorem 1.3 and 1.7},
there is an induced (not necessarily very) basic slc-trivial fibration  as
$f'\colon (X', B_{X'}) \to (Y',D')$
such that $D'=p^*D$ and
the following diagram commutes:
\begin{equation}\label{eq3.1}
\xymatrix{
(X', B_{X'})\ar[r]^-q \ar[d]_-{f'}& (X, B_{X})\ar[d]^-f \\ 
Y'\ar[r]_-p& Y
}
\end{equation}
 In particular, there is also a structure decomposition $D'=K_{Y'}+B_{Y'}+M_{Y'}$.
In general, $B_{Y'}$ is not effective, but is a sub boundary
and $B^{<0}_{Y'}$ is $p$-exceptional
 by \cite{fujino-slc-trivial}*{Theorem 1.7}.
Let $P$ be a prime divisor on $Y'$ and $t_P=\mult_P  B_{Y'}$.
Then, $K_{X'}+B_{X'}+(1-t_P)f'^*P$ is sub-slc over the generic point of $P$ by 
\cite{fujino-slc-trivial}*{(4.5)}. That is, 
for any prime divisor $Q \in \Supp (f'^*P_{\eta})$, 
there are
two rational numbers $a_Q:=\mult_Q B_{X'}$ and $b_Q:=\mult_Q {f'}^*P_{\eta}$
such that 
$$
a_Q+(1-t_P)b_Q\leq 1.
$$ 
Moreover, $a_Q=1$ implies that $t_P=1$.
Since the generic point of $P$ is Cartier, $b_Q$ is a positive integer for every $Q$. 
Therefore, 
\begin{equation}\label{eq3.2}
0\leq b_Q-1\leq -a_Q+t_Pb_Q.
\end{equation}
Let $R=f'^*(B^{<0}_{Y'})_{\eta}$. Then by \eqref{eq3.2}, we can check that
$\lceil -(B_{X'}^{<1})+R\rceil$ is effective.
Therefore, there is a natural injective morphism
$$
\alpha\colon \mathcal O_{Y'}\hookrightarrow
f'_*\mathcal O_{X'}(\lceil -(B_{X'}^{<1})+R\rceil).
$$
The second lemma shows that $\alpha$ is an isomorphism. More precisely,

\begin{lem}\label{res-lem}
Notation as above. Then
\begin{equation}\label{eq3.3}
\alpha\colon \mathcal O_{Y'}\hookrightarrow
f'_*\mathcal O_{X'}(\lceil -(B_{X'}^{<1})+R\rceil)
\end{equation}
is an isomorphism. In particular,  if
$B_{Y'}$ is effective,
then $f'\colon (X', B_{X'}) \to (Y',D')$ is a very basic slc-trivial fibration.
\end{lem}

\begin{proof}
Since $Y'$ is normal and $f'_*\mathcal O_{X'}(\lceil -(B_{X'}^{<1})+R\rceil)$ is torsion-free,
it suffices to prove that $\alpha$ is an isomorphism in codimension one.
Let $P$ be a prime divisor on $Y'$ not contained in $\Exc(p)$.
Then
$$
f'_*\mathcal O_{X'}(\lceil -(B_{X'}^{<1})+R\rceil)\simeq f'_*\mathcal O_{X'}(\lceil -(B_{X'}^{<1})\rceil)
$$
at the generic point of $P$.
Since $p\colon P\to p(P)$ is an isomorphism at the generic points,
$\alpha$ is an isomorphism at the generic point of $P$ by pulling back the isomorphism
$$
\mathcal O_Y\simeq f_*\mathcal O_X(\lceil -(B^{<1}_X)\rceil)
$$
at the generic point of $p(P)$.
Therefore, we assume that $P$ is a prime divisor in $\Exc(p)$ and $t_P=\mult_P  B_{Y'}$.
Then there exists a prime divisor $Q\in \Supp(f'^*P_{\eta})$, 
a rational number $a:=\mult_Q B_{X'}$ and a positive integer $b:=\mult_Q {f'}^*P_{\eta}$
such that $a+(1-t_P)b=1$.
Equivalently,
$$
\lceil - a+t_Pb \rceil=b-1.
$$
That is,
$\lceil -(B_{X'}^{<1})+R\rceil \ngeq f'^*P_{\eta}$ over the generic point of $P$. Then
$$
\mathcal O_{Y'} \hookrightarrow 
f'_*\mathcal O_{X'}(\lceil -(B_{X'}^{<1})+R\rceil) \subsetneq
 \mathcal O_{Y'}(P)
$$
at the generic point of $P$. It follows that $\mathcal O_{Y'} \simeq
f'_*\mathcal O_{X'}(\lceil -(B_{X'}^{<1})+R\rceil)$ at the generic point of $P$. 
In conclusion, $\alpha$ is isomorphic at the generic point of any prime divisor on $Y'$ and 
thus an isomorphism.
In particular, if $B_{Y'}$ is effective,
then $R=0$ and 
 $\mathcal O_{Y'} \simeq
f'_*\mathcal O_{X'}(\lceil -(B_{X'}^{<1})\rceil)$.
By the definition, $f'\colon (X', B_{X'}) \to (Y',D')$ is a very basic slc-trivial fibration.
\end{proof}

There are several useful corollaries of Lemma \ref{res-lem}.
One of them is that, this lemma enable us to modify $Y$ and thus $f$ by a higher model $Y'$ and $f'$
 whose discriminant part is still a boundary.
In particular, we have the following  corollary.

\begin{cor}\label{dlt-cor}
Let $f\colon (X, B_X) \to (Y,D)$ be a very basic slc-trivial fibration
with the structure decomposition $D=K_Y+B_Y+M_Y$.
Assume that $K_Y+B_Y$ is $\mathbb Q$-Cartier.
Let $p\colon Y'\to Y$ be a $\mathbb Q$-factorial
dlt blow-up such that $K_{Y'}+B=p^*(K_Y+B_Y)$
and $(Y', B)$ is a dlt pair.
Then there is an induced very basic slc-trivial fibration $f'\colon (X', B_{X'}) \to (Y',D')$.
\end{cor}

\begin{proof}
By Lemma \ref{res-lem},
it suffices to show that $B$ coincides with $B_{Y'}$
in the structure decomposition of $p^*D=K_{Y'}+B_{Y'}+M_{Y'}$.
Note that $M_{Y'}$ is a nef $\mathbb Q$-Cartier $\mathbb Q$-divisor and $M_Y=p_*M_{Y'}$ is also a 
nef $\mathbb Q$-Cartier $\mathbb Q$-divisor by assumption.
By the negative lemma (cf. \cite{kollar-mori}*{Lemma 3.39}), 
we have that $M_{Y'}+E=p^*M_Y$ for some effective $p$-exceptional
$\mathbb Q$-divisor $E$ on $Y'$.
Then 
$$
K_{Y'}+B+p^*M_Y=p^*D=K_{Y'}+B_{Y'}+M_{Y'}=K_{Y'}+B_{Y'}-E+p^*M_Y,
$$
which implies that $B+E=B_{Y'}$. 
Since $p$ is a dlt blow-up, $E$ is supported by $B^{=1}$.
It follows that $E=0$ since $B_{Y'}$ is a sub boundary.
In particular, $M_{Y'}=p^*M_Y$ and  $B=B_{Y'}$.
\end{proof}

In the rest of this section, we use Lemma \ref{res-lem} to generalize the connectedness properties in
\cite{liu}*{(2.2)}.

\begin{defn}\label{def-lc}
Let $f\colon (X, B_X) \to (Y,D)$ be a basic slc-trivial fibration.
Then an {\em{lc center}} $Z$ of $f$ is an image of some slc center of $(X, B_X)$
such that $Z\subsetneq Y$.
\end{defn}

\begin{rem}
As showed in Remark \ref{rem1.3}, after some modifications, a normal qlc pair is a 
 very basic slc-trivial fibration. Then the definition of qlc centers 
(cf. \cite{fujino-foundations}*{Definitions 6.2.2 and 6.2.8}) coincides
with the lc centers defined above. Note also that a crepant log resolution 
$f\colon (X,B_{X}) \to (Y,B_Y)$ of an lc pair $(Y,B_Y)$ is also a very basic lc-trivial fibration.
In this case, the definition of lc centers in Section \ref{2.2} also coincides with
the lc centers of $f$.
\end{rem}

Since every stratum of $X$ is dominant onto $Y$, the slc center of $(X, B_X)$ mapping into an lc center
$Z$ of $f$ must be contained in $\Supp B^{=1}_X$. By further blowing ups, we always assume that
every maximal slc center of $(X, B_X)$ mapping into $Z$ is a smooth component of $B^{=1}_X$
and the union of them is a simple normal crossing divisor on $X$.

\begin{cor}[Connectedness]\label{conn-cor}
Assume that there exists a commutative diagram 
$$
\xymatrix{
(X', B_{X'})\ar[r]^-q \ar[d]_-{f'}& (X, B_{X})\ar[d]^-f \\ 
Y'\ar[r]_-p& Y
}
$$ with the same notation as in \eqref{eq3.1}.
Let  $Z$ be a union of lc centers of $f$ 
and $Z'$ be the union of all lc centers of $f'$ mapping into $Z$. 
Assume that $Z'$ is a simple normal crossing divisor on $Y'$.
We also assume that 
the union of strata of $B^{=1}_X$ mapping into $Z$ is a simple normal crossing divisor
on $X$, denote it as $T$, and that the union of strata of $B^{=1}_{X'}$ mapping into $Z'$
is a simple normal crossing divisor on $X'$, denote it as $T'$. 
Put $K_T+B_T:=(K_X+B_X)|_T$ and $K_{T'}+B_{T'}:=(K_{X'}+B_{X'})|_{T'}$.
Then
there exists a commutative diagram 
$$
\xymatrix{
(T', B_{T'})\ar[r]^-q \ar[d]_-{f'|_{T'}}& (T, B_{T})\ar[d]^-{f|_T} \\ 
Z'\ar[r]_-p& Z
}
$$
such that $\mathcal O_{Z} \hookrightarrow (f|_{T})_*\mathcal O_{T}$ is an isomorphism
and
$\mathcal O_{Z'} \hookrightarrow (f'|_{T'})_*\mathcal O_{T'}$ is an isomorphism
at every generic points of $Z'$.
\end{cor}

\begin{proof}
The proof is similar to \cite{fl-plt}*{Lemma 4.1 and Corollary 4.2} and \cite{liu}*{Lemma 2.2}. 
First, we show that  $ \mathcal O_{Z}\simeq(f|_{T})_*\mathcal O_{T}$.
Consider the following short exact sequence 
$$
0\to \mathcal O_X(A-T)\to \mathcal O_X(A)\to \mathcal O_T(A|_T)\to 0
$$ 
where $A=\lceil -(B^{<1}_X)\rceil$ and $A|_T=\lceil -(B^{<1}_X)|_T\rceil=\lceil -(B^{<1}_T)\rceil$.
Then we obtain the long exact sequence 
\begin{equation*}
\begin{split}
0 &\longrightarrow 
f_*\mathcal O_X(A-T)\longrightarrow f_*\mathcal O_X(A)\longrightarrow 
f_*\mathcal O_T(A|_T)
\\ 
&\overset{\delta}\longrightarrow R^1f_*\mathcal O_X(A-T)\longrightarrow \cdots.  
\end{split}
\end{equation*}
Note that 
$$
A-T-(K_X+\{B_X\}+B^{=1}_X-T)=-(K_X+B_X)\sim _{\mathbb Q, f}0. 
$$ 
By \cite{fujino-foundations}*{Theorem 5.6.3}, 
every nonzero local section of the sheaf 
$R^1f_*\mathcal O_X(A-T)$ contains 
in its support the $f$-image 
of some slc stratum of $(X, \{B_X\}+B^{=1}_X-T)$. 
By the construction of $T$, 
no slc stratum of $(X, \{B_X\}+B^{=1}_X-T)$ are mapped 
into $Z$ by $f$. 
On the other hand, the support of $f_*\mathcal O_T(A|_T)$ is 
contained in $Z=f(T)$. 
Therefore, the connecting homomorphism $\delta$ is a zero map. 
Thus we get a short exact sequence 
$$
0\to f_*\mathcal O_X(A-T)\to \mathcal O_Y\to f_*\mathcal O_T(A|_T)\to 0. 
$$ 
Since $f_*\mathcal O_X(A-T)$ is contained in $\mathcal O_Y$ and 
$f(T)=Z$, 
we have $f_*\mathcal O_X(A-T)=\mathcal I_Z$, where $\mathcal I_Z$ 
is the defining ideal sheaf of $Z$ on $Y$. 
Thus, by the above short exact sequence, 
we obtain that the natural 
map $\mathcal O_Z\to f_*\mathcal O_T(A|_T)$ is an isomorphism. 
It follows that the natural map $\mathcal O_Z\to (f|_{T})_*\mathcal O_T$ 
is an isomorphism. 

Next, we show that $\mathcal O_{Z'}\hookrightarrow (f'|_{T'})_*\mathcal O_{T'}$
is an isomorphism at every generic points of $Z'$.
Consider the following short exact sequence 
$$
0\to \mathcal O_{X'}(A'-T')\to \mathcal O_{X'}(A')\to \mathcal O_{T'}(A'|_{T'})\to 0
$$ 
where $A'=\lceil -(B_{X'}^{<1})+R\rceil$. 
Then we obtain the long exact sequence 
\begin{equation*}
\begin{split}
0 &\longrightarrow 
f'_*\mathcal O_{X'}(A'-T')\longrightarrow f'_*\mathcal O_{X'}(A')\longrightarrow 
f'_*\mathcal O_{T'}(A'|_{T'})
\\ 
&\overset{\delta'}\longrightarrow R^1f'_*\mathcal O_{X'}(A'-T')\longrightarrow \cdots.  
\end{split}
\end{equation*}
Note that 
$$
A'-T'-(K_X+\{B_{X'}^{<1}-R\}+B^{=1}_X-T')=-(K_X+B_X)+R\sim _{\mathbb Q, f'}f'^*(B^{<0}_{Y'})_{\eta}.
$$
Similarly as above,  $\delta'$ is a zero map in codimension one of $Y'$.
By Lemma \ref{res-lem} and the assumption that $Z'$ is a simple normal crossing divisor, 
we get a short exact sequence 
$$
0\to \mathcal I_{Z'}\to \mathcal O_{Y'}\to f'_*\mathcal O_{T'}(A'|_{T'})\to 0
$$ 
at every generic points of $Z'$. That is, 
$\mathcal O_{Z'} \hookrightarrow f'_*\mathcal O_{T'}(A'|_{T'})$
is an isomorphism outside a codimension one subset of $Z'$. It follows that
 $\mathcal O_{Z'} \hookrightarrow f'_*\mathcal O_{T'}$ is an isomorphism at every generic points of $Z'$.
\end{proof}

The following corollary is a generalisation of \cite{liu}*{Corollary 2.3}.
\begin{cor}\label{ind-cor}
Notation as in Corollary \ref{conn-cor}. Assume further that $Z''$ is a normal component of $Z'$. 
Let $T_1$ (resp. $T_2$) be the union of irreducible components of $T'$ dominant onto
(resp. mapping into) $Z''$.
Put $K_{T_i}+B_{T_i}=(K_X+B_X)|_{T_i}$ by adjunction for $i=1,2$. 
Then 
$$
\rank (f'|_{T_1})_*\mathcal O_{T_1}(\lceil -(B^{<1}_{T_1})\rceil)=1 \quad  \text{and} \quad 
\mathcal O_{Z''}\simeq (f'|_{T_1})_*\mathcal O_{T_1}.
$$ 
In particular, $f'|_{T_1}\colon (T_1, B_{T_1}) \to (Z'', D'|_{Z''})$ is a basic slc-trivial fibration.
\end{cor}

\begin{proof}
Consider the following commutative diagram:
$$
\xymatrix{
T_1 \ar[d]_{g}\ar@{^(->}[r]^{\iota}& T_2\ar[d]^{f'}\\ 
\widetilde{Z} \ar[r]_{p}& Z''
}
$$
where $\iota \colon T_1\to T_2$ is the natural closed immersion 
and 
$$
\xymatrix{T_1 \ar[r]^{g}& \widetilde{Z} \ar[r]^p& Z''
}
$$ 
is the Stein factorization of $f'\circ \iota\colon T_1\to Z''$. 
By  Corollary \ref{conn-cor}, 
$ \mathcal O_{Z''} \hookrightarrow (f'|_{T_2})_*\mathcal O_{T_2}$ is an isomorphism
at the generic point of $Z''$.
Therefore, $\iota\colon T_1\to T_2$ is an isomorphism over the generic point of $Z''$.
In particular, $p$ is birational.
By \cite{fl-normal}*{Claim 1}, $p$ is the normalization.
Since $Z''$ is normal, $p$ is an isomorphism.
It follows that $\mathcal O_{Z''}\simeq (f'|_{T_1})_*\mathcal O_{T_1}$.
Note that
$$
(f'|_{T_1})_*\mathcal O_{T_1}(\lceil -(B^{<1}_{T_1})\rceil)
$$
is torsion-free, and isomorphic to $(f'|_{T_2})_*\mathcal O_{T_2}(\lceil -(B^{<1}_{T_2})\rceil$
at the generic point of $Z''$ by the same proof of 
\cite{fl-normal}*{Claim 2}.
It follows that  
$$
\rank (f'|_{T_1})_*\mathcal O_{T_1}(\lceil -(B^{<1}_{T_1})\rceil)=1.
$$
By Definition \ref{slc-def}, 
$f'|_{T_1}\colon (T_1, B_{T_1}) \to (Z'', D'|_{Z''})$ is a basic slc-trivial fibration.
\end{proof}

\begin{rem}
In general,
 $(f'|_{T_1})_*\mathcal O_{T_1}(\lceil -(B^{<1}_{T_1})\rceil)$
is not isomorphic to $\mathcal O_{Z''}$.
This is because the pushforward of $R$ on $Z''$ provides a gap between these two sheaves.
Therefore, $f'|_{T_1}$ is not necessary to be a very basic slc-trivial fibration.
\end{rem}

\begin{rem}\label{comm-rem}
In Corollary \ref{ind-cor}, the natural commutative diagram
$$
\xymatrix{
\mathcal O_{X'} \ar[d]_{f'}\ar[r]^{j}& \mathcal O_{T_1}\ar[d]^{f'|_{T_1}}\\ 
f'_*\mathcal O_{X'} \ar[r]_{j}&  f'_*\mathcal O_{T_1}
}
$$ 
induces a commutative diagram:
$$
\xymatrix{
\mathcal O_{X'} \ar[d]_{f'}\ar[r]^{j}& \mathcal O_{T_1}\ar[d]^{f'|_{T_1}}\\ 
\mathcal O_{Y'} \ar[r]_{j}&  \mathcal O_{Z''}
}
$$ 
Assume that a $\mathbb Q$-line bundle  $\varphi=0$ in $\Pic(X')\otimes \mathbb Q$.
Then $\varphi|_{T_1}=0$ in $\Pic(T_1)\otimes \mathbb Q$. By above  commutative  diagram,
$f'_*(\varphi):=\frac{1}{r}f'_*(r\varphi)=0$ in $\Pic(Y')\otimes \mathbb Q$
and 
$$
f'_*(\varphi)|_{Z''}=(f'|_{T_1})_*(\varphi|_{T_1})=0
$$ 
in $\Pic(Z'')\otimes \mathbb Q$
for some integer $r$ such that $r\varphi$ is a line bundle.
\end{rem}


\section{The numerically trivial moduli part}\label{sec-lemma}
Let $f\colon (X, B_X) \to (Y,D)$ be a very basic slc-trivial fibration
with the structure decomposition $D=K_Y+B_Y+M_Y$.
Assume that $M_Y$ is $\mathbb Q$-Cartier, equivalently, $K_Y+B_Y$ is  $\mathbb Q$-Cartier.
By Corollary \ref{cart-cor}, $(Y, B_Y)$ is log canonical.
Let $Z$ be a union of lc centers of $f$.
Then we have the following theorem, which is 
a generalisation of \cite{liu}*{Lemma 3.3}.

\begin{thm}\label{key thm}
If $M_Y|_Z\equiv 0$, then $M_Y|_Z\sim_{\mathbb Q}0$.
\end{thm}

For simplicity, we can assume that $Z$ is connected.
Take a log resolution  $g\colon Y' \to Y$ such that $\Supp g^{-1}(Z)$ is a 
simple normal crossing divisor on $Y'$. 
Let $D'=g^*D$ and $f'\colon (X',B_{X'})\to (Y', D')$ be the induced basic slc-trivial fibration
with the structure decomposition $D'=K_{Y'}+B_{Y'}+M_{Y'}$.
Let $Z'$ be the union of all lc centers of $f'$
mapping into $Z$.  Note that $Z'$ is contained in $B_{Y'}^{=1}$.
By Corollary \ref{ind-cor}, there is an induced basic slc-trivial fibration
on each smooth component of $Z'$.
Using the connectedness lemma on $g\colon Y' \to Y$, $Z'$ is connected and $g_*\mathcal O_{Z'}=\mathcal O_{Z}$.
As in Corollary \ref{dlt-cor}, $M_{Y'}+E=g^*M_Y$ where 
$E$ is an effective $g$-exceptional
$\mathbb Q$-divisor supported by $B_{Y'}^{<1}$. By assumption, $M_{Y'}|_{Z'}+E|_{Z'}\equiv 0$,
where $E|_{Z'}$ is an effective $\mathbb Q$-divisor on $Z'$. 
Since $M_{Y'}|_{Z'}$ is nef, it follows that
$E|_{Z'}=0$ and $M_{Y'}|_{Z'}\equiv 0$.
If we can show that $M_{Y'}|_{Z'}\sim_{\mathbb Q}0$, that is, 
$g^*\mathcal O_{Z}(rM_{Y})\simeq \mathcal O_{Z'}(rM_{Y'})\simeq \mathcal O_{Z'}$
for some positive integer $r$, then 
by projection formula,
$$
\mathcal O_{Z}\simeq g_* \mathcal O_{Z'}\simeq g_* g^*\mathcal O_{Z}(rM_Y)
\simeq g_*\mathcal O_{Z'}\otimes \mathcal O_{Z}(rM_Y) \simeq \mathcal O_{Z}(rM_Y).
$$
That is, $M_Y|_Z\sim_{\mathbb Q}0$.
 Therefore, in the rest of this section,
we will take log resolutions as high as we want by keeping Corollary \ref{ind-cor} in mind,
and prove Theorem \ref{key thm} from two different points of view.
Note that these two approaches are not totally separated.

\subsection{The viewpoint of trivial fibrations}\label{subsec4.1}

We can further assume that the morphism 
$f\colon (X, B_X) \to (Y,D)$
satisfies the following conditions (a)--(g),
stated also in \cite{fujino-slc-trivial}*{Proposition 6.3},  \cite{fujino-fujisawa-liu}*{Section 5}
and \cite{liu}*{Lemma 3.3}:
\begin{itemize}
\item[(a)] $Y$ is smooth and $X$ is simple normal crossing,
\item[(b)] $f$ is a projective surjective morphism,
\item[(c)] $\Sigma_X$ and $\Sigma_Y$ are simple normal crossing divisors on $X$ and $Y$ respectively,
\item[(d)] $B_X$ and $B_Y$,  $M_Y$ are supported by $\Sigma_X$ and $\Sigma_Y$ respectively,
\item[(e)] every stratum of $(X, \Sigma^h_X)$ is smooth over $Y^*:=Y\setminus \Sigma_Y$,
\item[(f)] $f^{-1}(\Sigma_Y)\subset \Sigma_X$, $f(\Sigma^v_X)\subset \Sigma_Y$,
\item[(g)] $(B^h_X)^{=1}$ is Cartier. 
\end{itemize} 
Let $Z:=B_Y^{=1}=\sum Z_i$ where $Z_i$ is an irreducible component of $Z$ for every $i$.
Let $T_i$ be the union of components of $B_X^{=1}$ 
dominant onto $Z_i$.
As stated in Remark \ref{rem1.2}, it is less confusing by using $\mathbb Q$-line bundle.
That is, we have an isomorphism in $\Pic (X) \otimes \mathbb Q$
$$
\varphi\colon f^*\mathcal D=f^*\mathcal O_Y(K_Y+B_Y+M_Y) \to  \mathcal O_X(K_X+B_X).
$$
Note that we can view $\varphi$ as a $\mathbb Q$-line bundle 
($=(f^*\mathcal D)^*\otimes \mathcal O_X(K_X+B_X)$) such that
$\varphi=0$ in $\Pic(X)\otimes \mathbb Q$ (see Remark \ref{comm-rem}).
By restricting sheaves on $T_i$ and adjunction, we have an induced isomorphism
$$
\varphi_i\colon f^*\mathcal D_i \to \mathcal O_{T_i}(K_{T_i}+B_{T_i})
$$
where $\varphi_i=\varphi|_{T_i}$ 
and $\mathcal D_i=\mathcal D|_{Z_i}$ for every $i$.
By Corollary \ref{ind-cor}, $f|_{T_i}\colon (T_i, B_{T_i}) \to (Z_i, \mathcal D_i)$
is a basic slc-trivial fibration with the structure decomposition
$$
\mathcal D_i=\mathcal O_{Z_i}(K_{Z_i}+B_{Z_i}+M_{Z_i})
$$
for every $i$. 
Let $\tau:=f_*(\varphi)$ and $\tau_i:=(f|_{T_i})_*(\varphi_i)$ as in Remark \ref{comm-rem}.

\begin{lem}\label{inh-lem}
Assume that $M_Y|_{Z_i}\equiv 0$. Then
$\tau_i\otimes\mathcal O_{Z_i}(M_{Z_i})=(\tau\otimes\mathcal O_Y(M_Y))|_{Z_i}$.
\end{lem}

\begin{proof}
By Remark \ref{comm-rem}, 
$(\tau\otimes \mathcal O_Y(M_Y))|_{Z_i}=\tau_i\otimes
\mathcal O_{Z_i}(M_Y)$.
Therefore, it suffices to show that $\mathcal O_{Z_i}(M_Y)=\mathcal O_{Z_i}(M_{Z_i})$.
By adjunction, we have 
$$
\mathcal O_{Z_i}(K_{Z_i}+B_{Z_i}+M_{Z_i})=\mathcal O_Y(K_Y+B_Y+M_Y)|_{Z_i}
=\mathcal O_{Z_i}(K_{Z_i}+B_Y-Z_i+M_Y).
$$
Let $B_i=(B_Y-Z_i)|_{Z_i}$.  
Note that $\Supp B_Y$ is simple normal crossing on $Y$ and  $\Supp B_i$
is simple normal crossing on $Z_i$.
Let $P$ be a prime divisor in $\Supp (B_Y-Z_i)$.
Then $P_i:=P|_{Z_i}$ is a union of prime divisors.
Let $t_P$ be a number such that $ (T_i, B_{T_i}+t_P(f|_{T_i})^*P_i)$ is sub-slc
over the generic points of $P_i$.
Then by inversion of adjunction, $(X, B_{X}+t_Pf^*P)$ is sub-slc at around the generic points of $P_i$.
This implies that
$$
B_i\leq B_{Z_i}
$$ 
by the definition of the discriminant part.
Therefore, there is a natural injective morphism
$$
\alpha\colon \mathcal O_{Z_i}(M_{Z_i})\hookrightarrow \mathcal O_{Z_i}(M_Y).
$$
Since $M_Y|_{Z_i}\equiv 0$ and $M_{Z_i}$ is nef, $\alpha$ has to be an identity.
\end{proof}

\begin{rem}
Note that the symbol ``$=$'' for two sheaves in Lemma \ref{inh-lem} and the rest of this paper does
not only mean to be equal in the $\mathbb Q$-Picard group, but also
 mean that the 
defining Cartier divisors in both sides are exactly the same, equivalently, the principle divisor of the difference
is given by some constant function.
\end{rem}

\begin{rem}
In general, another direction that $ B_{Z_i}\leq  B_i$ is not necessarily true.
This is because the discriminant part $B_Y$ is defined by the slc threshold 
over the generic point of $P$ on $Y$, 
which is not necessarily the slc threshold over $Z_i\cap P$. 
But in this case, we can take further blowing ups at $Z_i\cap P$
and consider the new relationship between $ B_{Z_i}$ and $B_i$.
We won't go into these details in this paper since Lemma \ref{inh-lem} is sufficient.
\end{rem}

By the inheriting nature showed in Lemma \ref{inh-lem}, 
we give the following definition:

\begin{defn}\label{can-mod}
Notation as above.
$\tau\otimes \mathcal O_Y(M_Y)$ is called the {\em{canonical moduli part}} of $f$
and denote it as $\mathcal M^c_Y$.
\end{defn}

\begin{rem}
Note that the canonical moduli part is defined only on a sufficiently high model of $(Y, D)$.
Actually, it is more convenient to define it as a b-divisor (cf. \cite{fujino-slc-trivial}*{Definition 2.10})
by considering $\tau$ as a b-divisor.
\end{rem}

By  \cite{fujino-slc-trivial}*{Lemma 7.3, Theorem 8.1}, for every $i$,
there exists a finite surjective morphism $\pi_i\colon W_i \to Z_i$ (unipotent reduction)
and an induced pre-basic slc-trivial fibration
$f'_i\colon (T'_i, B_{T'_i}) \to (W_i, \pi_i^*\mathcal D_i)$
such that 
$$
\pi_i^*\varphi_i\colon f'^*_i\pi_i^*\mathcal D_i \to \mathcal O_{T'_i}(K_{T'_i}+B_{T'_i}) 
$$
is an induced natural isomorphism.
Moreover, by \cite{fujino-slc-trivial}*{Proposition 6.3, Theorem 8.1}, 
the $\mathbb Q$-divisor $M_{W_i}:=\pi_i^* M_{Z_i}$ is Cartier and 
 $\pi_i^*\mathcal D_i=\mathcal O_{W_i}(K_{W_i}+B_{W_i}+M_{W_i})$ is the structure decomposition.
Let $\mathcal N_i$ be the eigensheaf defined in \cite{fujino-slc-trivial}*{(6.11)}.
By the Claim in the proof of \cite{fujino-slc-trivial}*{Proposition 6.3},
we have that

\begin{lem}[\cite{fujino-slc-trivial}*{Proposition 6.3}]\label{fujino-lem}
$\pi_i^*\mathcal M^c_{Z_i}=\mathcal M^c_{W_i}= \mathcal N_i$.
\end{lem}

By the assumption and Lemma \ref{inh-lem}, 
$\mathcal M^c_{W_i}\equiv 0$. Then in this case, we have that

\begin{lem}\label{cons-lem}
$\mathcal N_i^{\otimes r_i}=\mathcal O_{W_i}$ for some positive integer $r_i$.
\end{lem}

\begin{proof}
From the proof of \cite{fujino-fujisawa-liu}*{Theorem 1.3}, we see that
$\mathcal N_i$ is a canonical extension of a local subsystem of a polarizable 
variation of $\mathbb Q$-Hodge structures (see also Lemma \ref{mod-hodge} below). 
By \cite{deligne}*{Corollaire (4.2.8) (iii) b)},
there is a positive integer $r_i$ such that  $\mathcal N_i^{\otimes r_i}$ is constant.
It follows that $\mathcal N_i^{\otimes r_i}=\mathcal O_{W_i}$.
Or in another way, we can cut $W_i$ by general hyperplanes and reduce to the case that $W_i$ is a curve.
More precisely, for a general  hyperplane $H_i \subset W_i$,
$$
\mathcal N_i^{\otimes r_i}|_{H_i}=\mathcal O_{H_i} \quad \text{if and only if} \quad
\mathcal N_i^{\otimes r_i}=\mathcal O_{W_i}
$$
since $H_i$ is general.
So it suffices to prove the lemma for $\dim W_i=1$, which is \cite{liu}*{Lemma 3.3}.
\end{proof}

Combining Lemma \ref{fujino-lem} and Lemma \ref{cons-lem}, we immediately get that

\begin{cor}\label{liu-lem}
There exists a positive integer $r_i$ such that $(\mathcal M^c_{Z_i})^{\otimes r_i}=\mathcal O_{Z_i}$.
\end{cor}

Then by Lemma \ref{inh-lem} and Corollary \ref{liu-lem}, 
there exists a positive integer $r$ such that
\begin{equation}\label{eq4.1}
(\mathcal M^c_{Y})^{\otimes r}|_{Z_i}=\mathcal O_{Z_i}
\end{equation}
for every $i$. Therefore, $(\mathcal M^c_{Y})^{\otimes r}|_{Z}=\mathcal O_{Z}$ by patching together.
This proves Theorem \ref{key thm}.

\subsection{The viewpoint of period maps}

Let $f\colon (X, B_X) \to (Y,D)$ be a basic slc-trivial fibration
satisfying the conditions (a)--(g) in above subsection
and $\pi\colon W\to Y$ be the unipotent reduction with the induced pre-basic slc-trivial fibration
$h \colon (V, B_{V}) \to (W, \pi^*D)$.
Let $\Sigma_W:=\pi^{-1}(\Sigma_Y)$ and $\Sigma_V:=h^{-1}(\Sigma_{W})$.
 Replacing $h$ by a
cyclic cover of the generic fiber of $h$ (denote the cyclic cover group as $G$) and a further resolution 
(see \cite{fujino-slc-trivial}*{Section 6}), we get that 
$$
\mathcal V_{0}:=R^{\dim V-\dim W}(h|_{V_0})_*\iota_! \mathbb Q_{V_0\setminus (B^h_{V_0})^{=1}}
$$
underlies a graded polarizable admissible variation of $\mathbb Q$-mixed Hodge structure on $W_0$
which is unipotent at around $\Sigma_{W}$, where 
$V_0=V\setminus \Sigma_V$, $W_0=W\setminus \Sigma_W$ and $B_{V_0}=B_{V}|_{V_0}$.
By \cite{fujino-fujisawa}*{Theorem 7.1 and 7.3},
$$
h_*\omega_{V/W}(B_{V}^{=1})\simeq \Gr^0_F((\mathcal V)^*)
$$
where $\mathcal V$ is the canonical extension of $\mathcal V_0$.
By \cite{fujino-slc-trivial}*{Proposition 6.3},
 $h_*\omega_{V/W}(B_{V}^{=1})$ has an eigensheaf $\mathcal N\simeq \mathcal O_{W}(M_{W})$.
Using Definition \ref{can-mod} as in Lemma \ref{fujino-lem},
we  actually have that $\mathcal N=\mathcal M^c_{W}$.  Moreover,
 
\begin{lem}\label{mod-hodge}
On the open set $W_0$, the canonical moduli part
$$
\mathcal M^c_{W_0}=\mathcal M^c_{W}|_{W_0}\subset F^0\Gr_l^W((\mathcal V_0)^*)
$$
is a direct summand of the lowest piece of the Hodge filtration of $\Gr_l^W((\mathcal V_0)^*)$
for some integer $l$.
\end{lem}

\begin{proof}
The same as Step 4 in the proof of  \cite{fujino-fujisawa-liu}*{Theorem 1.3},
it suffices to show that the action of the cyclic cover group $G$ on $\Gr_F^0((\mathcal V_0)^*)$
preserves its weight filtration.
Note that the action of $G$ on $\Gr_F^0((\mathcal V_0)^*)$ does not commute 
the stalks of different points of $W_0$.
 That is, it suffices to show that the action of $G$
preserves the weight filtration of $\Gr_F^0((\mathcal V_0)^*)|_w$
on each point $w\in W_0$. 
Since $h$ is log smooth over $ W_0$, by \cite{fujino-fujisawa}*{Remark 4.7},
we can assume that $w\in W_0=\Delta^*$ and $W=\Delta$ (cf. \cite{fujino-fujisawa}*{Lemma 4.10 and 4.12}).
Then our conclusion follows from \cite{fujino-fujisawa-liu}*{Lemma 4.7}.
\end{proof}

Let $\mathcal H=\Gr_l^W((\mathcal V_0)^*)_{\prim}$ be the polarizable variation of $\mathbb Q$-Hodge structure
containing $\mathcal M^c_{W_0} $.
Note then that the Hodge filtration on $\mathcal H$ is 
(cf. \cite{fujino-fujisawa}*{Remark 3.15} and \cite{fujino-fujisawa-liu}*{Remark 4.6}):
$$
0\subset F^0 \subset F^{-1}\subset \cdots \subset F^{-p} \subset \cdots \subset F^{-n+1}\subset F^{-n}=\mathcal H.
$$
By the first Riemann bilinear relation of its $\mathbb Q$-polarization $Q$,
$\mathcal H$ is determined by
\begin{equation*}
0\subset F^0 \subset  \cdots \subset F^{-m}, \quad m=\lfloor \frac{n-1}{2}\rfloor.
\end{equation*}
Fix a particular  $\mathbb Q$-Hodge structure $H_o$ corresponding to some point $o \in W_0$.
We drop the subscript of $H_o$ if there is no confusion and denote $H_{\mathbb C}$  
(resp. $H_{\mathbb R}$, $H_{\mathbb Z}$) as
the $\mathbb C$-structure (resp. $\mathbb R$-structure, 
integral lattice) of $H_o$.
Let $G(k, H_{\mathbb C})$ be a Grassmannian that parameterizes $k$-dimensional subspaces of $H_{\mathbb C}$.
Consider the set $\breve D$ of all Hodge filtrations of weight $n$ 
with a sequence of fixed positive integers $k^{i}=\dim F^i$ and
satisfying the first Riemann bilinear relation. Then there is a projective embedding
 (cf. \cite{griff-2}*{Proposition 8.2})
$$
\iota \colon \breve D \to G(k^0, H_{\mathbb C})\times \cdots \times G(k^{-m}, H_{\mathbb C}).
$$
It follows that
$\breve D$ is a complete and projective algebraic variety. Let 
$p_0 \colon G(k^0, H_{\mathbb C})\times \cdots \times G(k^{-m}, H_{\mathbb C}) \to G(k^0, H_{\mathbb C})$
be the first projection. We denote 
$$
\breve D_0:=p_0\circ\iota (\breve D)
$$
as the image of the first projection of $\breve D$.

By Lemma \ref{mod-hodge}, the polarization $Q$ induces a bilinear form on 
$\mathcal M^c_{W_0}\oplus \overline{\mathcal M^c_{W_0}}$ such that
\begin{align}
&\  Q(u,u)=0 \quad\quad \text{if $u \in \mathcal M^c_{w}$ or $\bar u\in \mathcal M^c_{w}$,}\\
&\ Q(C u,\bar u)>0 \quad  \text{if $u \neq 0 \in \mathcal M^c_{w}$}
\end{align}
for every point $w\in W_0$, where  $\mathcal M^c_{w}=\mathcal M^c_{W_0}|_w$ and
$C$ is the Weil operator defined by 
$Cu=i^{n}u$ for $u \in \mathcal M^c_{w}$.
Note that (4.2)  is corresponding to the first Riemann bilinear relation and
(4.3) is  corresponding  to the second Riemann bilinear relation.
Let  $\ell_0:=\mathcal M^c_{o}$ be the $1$-dimensional subspace of $H_{\mathbb C}$.
Then a candidate definition of the period domain for the canonical moduli part is the following:
\begin{defn}
The {\em{period domain of the canonical moduli part}} 
(or {\em{classifying space of the canonical moduli part}}) is defined as
$$
D^c=\{\ell \in G(1, H_{\mathbb C})\mid \ell \subset S \text{ for some } S\in \breve D_0,
 ~ Q(C\ell, \bar \ell)>0  \}
$$
and
$$
\breve D^c=\{\ell \in G(1, H_{\mathbb C}) \mid \ell \subset S \text{ for some } S\in \breve D_0\}
$$
as the {\em{compact dual}} of $D^c$.
\end{defn}

Obviously, $\breve D^c$ contains in $\{\ell \in G(1, H_{\mathbb C}) \mid Q(\ell, \ell)=0 \}$,
which is a quadric hypersurface 
 of $G(1, H_{\mathbb C})\cong \mathbb P(H_{\mathbb C})$.
Let $D_1$ be the set of flags $\ell \subset F^0\subset H_{\mathbb C}$
and $\iota_1: D_1 \to G(1, H_{\mathbb C})\times G(k^0, H_{\mathbb C})$  be the projective embedding.
Then  $\breve D^c=p_1(\iota_1^{-1}p_2^{-1}(\breve D_0))$ where $p_1$ and $p_2$ are the
first and second projections of $\iota_1$, which implies that $\breve D^c$ is  a projective algebraic variety.

The orthogonal group of the bilinear form $Q$ is a linear algebraic group, defined over $\mathbb Q$. 
Let 
$$
G_{\mathbb C}=\{g\in \GL(H_{\mathbb C})~|~ Q(gu,gv)=Q(u,v) \text{ for all } u, v\in H_{\mathbb C}\}
$$ 
be the $\mathbb C$-rational points and
$$
G_{\mathbb R}=\{g\in \GL(H_{\mathbb R})~|~ Q(gu,gv)=Q(u,v) \text{ for all } u, v\in H_{\mathbb R}\}
$$
the $\mathbb R$-rational points respectively.
The same as the classical discussions under the condition that $\ell \subset S \text{ for some } S\in \breve D_0$
(cf. \cite{griff-1}*{Theorem 1.26}),
 $G_{\mathbb C}$ (resp. $G_{\mathbb R}$) acts on $\breve D^c$ (resp. $D^c$)
transitively. 
Therefore, $\breve D^c$
is nonsingular and $D^c\subset \breve D^c$ is open in the Hausdorff topology of $\breve D^c$
inheriting the structure of complex manifold. 
Moreover, 
$$
\breve D^c \cong G_{\mathbb C}/B, \text{ where } B=\{g\in  G_{\mathbb C}~|~ g\ell_0=\ell_0 \}
$$
and
$$
D^c \cong G_{\mathbb R}/(G_{\mathbb R}\cap B)
$$
where the embedding  $D^c\subset \breve D^c$ corresponds to the inclusion
$G_{\mathbb R}/(G_{\mathbb R}\cap B)\subset G_{\mathbb C}/B$.
Since $Q$ is assumed to take rational values on the lattice $X_{\mathbb Z}$, 
$$
G_{\mathbb Z}=\{g\in \GL(H_{\mathbb R})~|~ g H_{\mathbb Z}=H_{\mathbb Z}\}
$$
lies in $G_{\mathbb R}$ as an {\em{arithmetic}} subgroup.
The action of $G_{\mathbb Z}$ on $D^c$ is  properly discontinuous. Therefore,
for any subgroup $\Gamma\subset G_{\mathbb Z}$, the quotient of the complex structure of 
$D^c$ by $\Gamma$ turns $D^c/ \Gamma$ into a complex analytic variety. 
Now we are ready to define the period map induced by  the canonical moduli part:

\begin{defn}
Induced by the data above, the natural holomorphic mapping 
$$
p^c \colon W_0 \to D^c/ \Gamma
$$
is called the {\em{period map of the canonical moduli part}}.
\end{defn}

If $h$ is a basic slc-trivial fibration, or more generally,
if $\rank h_*\omega_{V/W}(B_{V}^{=1})=1$, then $\breve D^c=\breve D_0$
and $p^c =p_0\circ\Phi$ where $\Phi$ is the usual period map and $p_0$ is the first projection.
But even in this case, $\breve D^c=\breve D_0$ is not necessarily a Hermitian symmetric domain.
For example, $h$ is a family of Calabi--Yau 3-folds. So
we don't have the Baily--Borel--Satake compactification of $D^c/ \Gamma$ in general.
Moreover, $D^c/ \Gamma$ is not necessary to be algebraic.

On the other hand,  in case $\Gamma$ is arithmetic,
Kato--Nakayama--Usui \cite{ku} have constructed an analogue of the toroidal compactification 
in the horizontal directions of $D/ \Gamma$. Recently,
Bakker--Brunebarbe--Tsimerman \cite{bbt} have proved that  $\Phi(W_0)$
is quasi-projective and Green--Griffiths--Laza--Robles \cite{gglr} tried to
strengthen the results of \cite{bbt} by the construction of a canonical
projective compactification of  $\Phi(W_0)$.
These evidences and efforts enlighten us to consider the 
compactification of $p^c(W_0)$ by patching strata of period maps of the canonical moduli part.
Anyway, we post the conjecture below following from \cite{griff-3}*{Conjecture 10.5}
and \cite{gglr}*{Conjecture 1.2.2}:

\begin{conj}\label{exist-conj}
Notation as above. Let $P_0=p^c(W_0)$. Then
there exists a projective compactification $P$ of $P_0$ such that 
\begin{itemize}
\item[(i)] $P$ admits a structure of a normal complex analytic variety.
\item[(ii)] There exists an analytic extension $p^c_e\colon W \to P$.
\item[(iii)] There exists an integer $r$ 
such that $(\mathcal M^c_W)^{\otimes r}=(p^c_e)^*\mathcal O^{an}_{P}(1)$.
\end{itemize}
\end{conj}

We prove Conjecture \ref{exist-conj} for a special case.

\begin{thm}\label{exist-thm}
If $\mathcal M^c_{W}\equiv 0$, then Conjecture \ref{exist-conj} holds true.
\end{thm}

\begin{proof}
As in Lemma \ref{cons-lem} and \ref{mod-hodge}, $\mathcal M^c_{W_0}$ is a local subsystem of $\mathcal H$.
This is equivalent to say that $P_0=p^c(W_0)$ is a point of $D^c/ \Gamma$
and $(\mathcal M^c_{W_0})^{\otimes r}=(p^c_e)^*(\mathcal L|_P)$ for some integer $r$ 
where $\mathcal L$ is the universal subbundle. Obviously $P_0=P$ is a bounded symmetric domain
and we can trivially extend $p^c$ to $p^c_e\colon W \to P$.
In particular, by taking the canonical extension we have
$$
(\mathcal M^c_{W})^{\otimes r}=(p^c_e)^*\mathcal O^{an}_{P}
$$
where $\mathcal O^{an}_{P}=\mathcal O^{an}_{P}(1)=(\mathcal L|_P)=\mathbb C(P)$.
\end{proof}

Theorem \ref{exist-thm} allows us to patch different $\mathcal M^c_{W_i}\equiv 0$ on $W_i$ together 
into a  union $W':=\cup W_i$ if  there is an identity
\begin{equation}\label{eq4.4}
\mathcal M^c_{W_i}|_{W_i\cap W_{j}} = \mathcal M^c_{W_j}|_{W_i\cap W_{j}}
\end{equation}
for any $i, j$. Note that $W'$ is not necessary to be equidimensional.
Actually, \eqref{eq4.4} implies 
an identity of the images of period maps $P_i$ and $P_j$ for any $i, j$. 
Assume that $W'$ is connected. Then there is only one point $P=P_i=P_j$ for any $i, j$
and the period maps $(p^c_e)_i$ can be patched together giving 
$$
p^c_e\colon W'\to P
$$
such that $(p^c_e)^*\mathcal O^{an}_{P}|_{W_i}=(\mathcal M^c_{W_i})^{\otimes r_i}$ for some integer
$r_i$ for every $i$.
This gives another viewpoint of \eqref{eq4.1}.
Then to prove Theorem \ref{key thm}, the rest is the same as Subsection \ref{subsec4.1}
by descending the global sections of $(p^c_e)^*\mathcal O^{an}_{P}$
onto $Z$.

\begin{rem}
Conjecture \ref{exist-conj} holds true directly implies that the
b-semi-ampleness conjecture (see \cite{fujino-slc-trivial}*{Conjecture 1.4}) 
holds true. We hope that
the rich structures given by  arithmetic and representation theory on the period domain 
could shed some lights on the abundance of the moduli part.
\end{rem}

\begin{rem}
Conjecture \ref{exist-conj} also holds true if  $\breve D^c=\breve D_0$ and $D^c$ is in 
the {\em{classical case}} (cf. \cite{gglr}*{Remark 1.1.2}), that is,
 $ D^c$ is a Hermitian symmetric domain and geometrically corresponding to
period maps for abelian varieties or K3-type objects
(e.g. K3's, hyper-K\" ahler manifolds, cubic 4-folds).
For example, if $h$ is a family of K3 surfaces, then the projective compactification
$P$ contains in the first projection of the Baily--Borel--Satake compactification of $D/ \Gamma$
and the semi-ampleness of the canonical moduli part is given by
\cite{fujino-bundle}*{Theorem 1.2}.
\end{rem}


\section{Proof of Theorem \ref{main-thm}}\label{sec5}
We discuss various cases according to the nef dimension $n(D)$.
Note that $(Y, B_Y)$ is 
log canonical by Corollary \ref{cart-cor} and 
the moduli $\mathbb Q$-divisor $M_Y$ is nef and $\mathbb Q$-Cartier by
Proposition \ref{moduli-nef}.

\begin{case}\label{case1}
Assume that $n(D)=0$. That is, $D= K_Y+B_Y+M_Y\equiv 0$.
We run the $(K_Y+B_Y)$-MMP (cf. \cite{fujino-surfaces}*{Theorem 1.1})
which is a sequence of extremal contractions and terminates at a model $(Y^*, B_{Y^*})$
such that one of the following holds:
\begin{itemize}
\item[(1)] $K_{Y^*}+B_{Y^*}$ is semi-ample.
\item[(2)] There is a morphism onto a curve $g\colon Y^*\to C$ such that $-(K_{Y^*}+B_{Y^*})$ is $g$-ample.
\item[(3)] $-(K_{Y^*}+B_{Y^*})$ is ample.
\end{itemize} 
Let $\pi\colon Y\to Y^*$ be the composition of the sequence of extremal contractions.

First, we assume that $K_{Y^*}+B_{Y^*}$ is semi-ample. Since $D$ is $\mathbb Q$-Cartier and $D\equiv 0$, 
there is a $\mathbb Q$-Cartier $\mathbb Q$-divisor
$D^*$ such that $D=\pi^*D^*$ by the cone and contraction theorem (cf. \cite{fujino-surfaces}*{Theorem 3.2}). 
In particular, $D^*\equiv 0$ on $Y^*$.
Let $M_{Y^*}=D^*-K_{Y^*}-B_{Y^*}$. 
Then it is easy to see that $M_{Y^*}=\pi_* M_Y$ is nef.
It follows that $K_{Y^*}+B_{Y^*}\equiv M_{Y^*} \equiv 0$. 
The first part $K_{Y^*}+B_{Y^*} \equiv 0$  implies that 
$K_{Y^*}+B_{Y^*} \sim_{\mathbb Q} 0$ since it is semi-ample.
As in Corollary \ref{dlt-cor}, we have $M_Y+E=\pi^*M_{Y^*}\equiv 0$ 
for some effective $\pi$-exceptional $\mathbb Q$-divisor $E$.
It follows that $E=0$ and  $\pi=\id$.
Then by \cite{fujino-fujisawa-liu}*{Theorem 1.3}, $D^*\sim_{\mathbb Q}M_{Y^*} \sim_{\mathbb Q} 0$.

Next, we assume that there is a morphism onto a curve $g\colon Y^*\to C$ 
such that $-(K_{Y^*}+B_{Y^*})$ is $g$-ample.
By the cone and contraction theorem again, there is a $\mathbb Q$-Cartier $\mathbb Q$-divisor
$D_C$ such that $D=\pi^*g^*(D_C)$. By Lemma \ref{slc-lem},
$g\circ \pi \circ f\colon (X, B_X) \to (C,D_C)$ is
an induced very basic slc-trivial fibration
such that $D_C\equiv 0$. 
Let $0\equiv D_C=K_C+B_C+M_C$ be the structure decomposition.
Then $C$ is either $\mathbb P^1$ or a smooth elliptic curve since $B_C$ is effective and $M_C$ is nef. 
If $C=\mathbb P^1$, then it is obvious that $D_C\sim _{\mathbb Q}0$ and thus $D\sim _{\mathbb Q} 0$. 
If $C$ is a smooth elliptic curve, 
then $B_C=0$ and $D_C\sim M_C\equiv 0$. 
By \cite{fujino-fujisawa-liu}*{Theorem 1.3} again, $M_C\sim _{\mathbb Q}0$
and thus $D\sim _{\mathbb Q} 0$.

Finally, we assume that there is a morphism onto a  point  $g\colon Y^* \to P$
such that $-(K_{Y^*}+B_{Y^*})$ is ample.
Then by  the cone and contraction theorem again,
$D\sim_{\mathbb Q} 0$.

\medskip

Anyway, we prove that if $n(D)=0$, then $D\sim_{\mathbb Q} 0$.
\end{case}

\begin{case}\label{case2}
Assume that $n(D)=1$.
By \cite{b8}*{Proposition 2.11 and 2.4.4}, 
there exists a morphism $g\colon Y \to C$ mapping onto a smooth curve $C$
and an ample $\mathbb Q$-divisor $A$ on $C$ such that $D\equiv g^*A$.
Let $F$ be a general fiber of $g$. Then
$$
0\equiv D|_F=(K_Y+B_Y+M_Y)|_F=K_F+B_F+M_Y|_F
$$
where $B_F=B_Y|_F$ is effective and $M_Y|_F$ is nef.
It follows that either 
 $\deg K_F<0$ or $\deg K_F=0$.
In the former case, $F\simeq {\mathbb P}^1$
and thus $D|_F\sim_{\mathbb Q} 0$.
In the latter case, $F$ is a smooth elliptic curve,
$B_F=0$ and $M_Y|_F\equiv 0$.  Let $T=f^*F$ on $X$.
Then $f \colon (X, B_X+T) \to (Y, D+F)$ with
$$
K_X+T+B_X\sim_{\mathbb Q} f^*(K_Y+F+B_Y+M_Y)
$$ 
is an induced very basic slc-trivial fibration since $\Supp B_Y\cap F=\emptyset$.
Note that $F$ is an lc center of the induced $f$. By Theorem \ref{key thm}, $M_Y|_F\sim_{\mathbb Q} 0$.
It follows that $D|_F\sim_{\mathbb Q} 0$.
Therefore, we have $\kappa (F, D|_F)=0$ in both cases.
By the canonical bundle formula 
(cf. \cite{mori}*{Theorem 1.12}),
there exists a $D_C$ on $C$ such that 
$$
D\sim_{\mathbb Q}g^*D_C+E
$$ 
where $E$ is an effective $\mathbb Q$-divisor supported by some special fibers of $g$.
Note that $D_C$ is ample since $D_C\equiv A$ by pushforward.
Then $1=n(D)\geq \kappa (D)\geq \kappa (g^*D_C)=\kappa (D_C)=1$.
That is, $\kappa (D)=n(D)=1$.
By \cite{fujita}*{Theorem 4.1}, $D$ is semi-ample.
\end{case}

\begin{case}\label{case3}
Assume that $n(D)=2$.
If $D$ is big, then $D$ is semi-ample 
by a suitable dlt blow-up as in Corollary \ref{dlt-cor} and \cite{liu}*{Corollary 3.5}.
If $M_Y$ is big, then $2D-(K_Y+B_Y)=D+M_Y$ is nef and big,
and $D$ is also semi-ample 
by Corollary \ref{dlt-cor} and \cite{liu}*{Remark 3.6}.
Therefore, we further assume that $D$ and $M_Y$ are not big. That is, $D^2=M_Y^2=0$.
We follow the discussions in \cite{ambro-nef}*{Section 6} and 
\cite{fujino-surfaces}*{Section 6} whose results heavily depended on \cite{sakai} and \cite{fujita}.

Let $\nu\colon V\to Y$ be a Sakai minimal resolution with respect to $D$ 
(see \cite{fujita}*{Definition 3.3} for example).
We put $K_V+B_V=\nu^*(K_Y+B_Y)$, $M_V=\nu^*M_Y$ and $D_V=\nu^*D$.
Note then that $D_V^2=M_V^2=0$ and $D$ is semi-ample if and only if $D_V$ is semi-ample.
Note also that there is an induced very basic slc-trivial fibration on $V$ by Lemma \ref{res-lem}.
Therefore, there is no harm to replace $Y$ by its Sakai minimal resolution $V$.
Then the same as the proof of \cite{ambro-nef}*{Theorem 0.3}, 
$K_Y+tD$ is nef and $(K_Y+tD)^2\geq 0$ for some $t$.
Then either $K_Y+tD$ is big or $K_{Y}^2=K_{Y}\cdot D=D^2= 0$.
In the former case, 
$$
aD-(K_Y+B_Y)=(a-1)D+M_Y=(a-2)D+K_Y+B_Y+2M_Y
$$
for any $a>2+t$.
Then $aD-(K_Y+B_Y)$ is nef and big since $(a-1)D+M_Y$ is nef 
and $K_Y+(a-2)D+B_Y+2M_Y$ is big. Therefore, $D$ is semi-ample 
 by \cite{liu}*{Remark 3.6} again. 
In the latter case,  we have  $B_Y\cdot D=M_Y\cdot D=0$ 
since $D$ is nef and 
$\mathbb Q$-linearly equivalent to an effective divisor by \cite{ambro-nef}*{Proposition 6.1 (4)}.
Since $D+M_Y$ has maximal nef dimension (cf. \cite{ambro-nef}*{Remark 2.4}), 
$M_Y\cdot K_Y=(D+M_Y)\cdot K_Y\geq 0$ by \cite{ambro-nef}*{Proposition 6.1 (2)}.
It follows that $M_Y\cdot K_Y =M_Y\cdot B_Y=0$.
Using \cite{ambro-nef}*{Proposition 6.1} again, we have that $\kappa(D)=\kappa(M_Y)= 0$.

We further assume that $H^1(Y, \mathcal O_Y)=0$. 
Then by \cite{sakai}*{Proposition 4} (see also 
\cite{ambro-nef}*{Theorem 6.2} 
or the proof of \cite{fujino-surfaces}*{Theorem 6.3}),
$Y$ is a degenerate del Pezzo surface and 
there exist positive rational numbers $a$, $b$ and $c$
such that $b+c>a$ and
$$
K_Y\sim -aT, \quad B_Y \sim bT, \quad M_Y \sim cT , \quad D\sim (b+c-a)T
$$
where $T$ is {\em{indecomposable of canonical type}} in the sense of
Mumford (cf. \cite{fujino-surfaces}*{Lemma 6.3} or \cite{fujita}*{Lemma 5.5}).
We will derive a contradiction assuming  $T\neq 0$. 
Note that $(Y, B_Y)$ is log canonical.
Assume that $\Supp T$ is reducible.
Then all irreducible components of $\Supp T$ are smooth rational curves 
with self intersection $-2$.
By adding a multiple of $T$ to $B_Y$,  we can assume that 
$(Y, B_{Y})$ is lc but not klt.
Replacing by a further blowing up if necessary (though not Sakai minimal anymore),
 we can further assume that
some smooth rational curve $C \subset \Supp T$ is an lc center of $(Y, B_{Y})$.
We run the $(K_Y+B_Y-C+M_Y)$-MMP. Let $C'\in \Supp (B_Y-C)$ be a smooth
rational curve such that $C'\cdot C>0$. Then 
$C'$ stays in an extremal ray  and
$(K_Y+B_Y-C+M_Y)\cdot C'=(D-C)\cdot C'=-C\cdot C'<0$.
Therefore, we get a morphism $\pi_1\colon Y\to Y_1$
contracting $C'$ with  $D_1=\pi_{1*}D$, $M_{Y_1}=\pi_{1*}M_Y$, $T_1=\pi_{1*}T$
and $C_1=\pi_{1*}C$.
Since $D\cdot C'=M_Y\cdot C'=T\cdot C'=0$, we have that $D=\pi_1^*D_1$, $M_Y=\pi_1^*M_{Y_1}$
and $T=\pi_1^*T_1$ by the cone and contraction theorem.
In particular, 
$T_1$ is a {\em{curve of fibre type}} in the sense of \cite{sakai}*{Section 2} 
(not necessarily indecomposable of canonical type) and $D_1\sim_{\mathbb Q}M_{Y_1}\sim_{\mathbb Q} T_1$.
Repeatedly, we finally get a minimal model $(Y^*, B_{Y^*})$ 
with that  $D^*=\pi_*D$, $B^*=\pi_*B_Y$, $M^*=\pi_*M_Y$, $T^*=\pi_*T$
and $C^*=\pi_*C$, where
$\pi: Y\to Y^*$ is the composition of extremal contractions.  
Note that $(T^*)^2=0$ and $T^*$ is irreducible. Hence
$$
D^*\sim_{\mathbb Q}M^*\sim_{\mathbb Q} T^*\sim_{\mathbb Q} C^*=B^*.
$$
By Lemma \ref{slc-lem}, there is an induced very basic slc-trivial fibration on $(Y^*, D^*)$
and $C^*$ is an lc center.
By $M^*\cdot C^*=0$ and Theorem \ref{key thm}, $C^*|_{C^*}$ and $T^*|_{T^*}$ are torsions.
It is a contradiction since $T|_T$ is not a torsion by  \cite{sakai}*{Proposition 5}.

Finally we assume that $H^1(Y, \mathcal O_Y)\neq 0$.
Then $Y$ is an irrational ruled surface over an elliptic curve $E$, of
type $II_c$ or $II_c^*$ in Sakai's classification table \cite{sakai2} 
(see also \cite{ambro-nef}*{Theorem 6.2 (ii)}).
Let $\alpha\colon Y\to E$ be the Albanese fibration.
Note then that $D=d_0E_0+d_1E_1$  where $E_0$  
and $E_1$ are disjoint sections of $\alpha$ in type $II_c$
and $D=d_0E_0$ in type $II_c^*$. In both cases, $B_Y$ is supported by
$E_0$ and $E_1$. As above,  we add a multiple of $D$ and assume that 
$(Y, B_{Y})$ is log canonical and $C$ is an lc center of $(Y, B_{Y})$ supported by 
$E_0$ or $E_1$. By Lemma \ref{slc-lem} and Theorem \ref{key thm} again,
there is an induced very basic slc-trivial fibration on $C$ and $C|_C$ is a torsion.
But the same as \cite{fujita}*{(5.9)}, $C|_C$ is not a torsion. Again, we get a contradiction.
\end{case}
\setcounter{case}{0}


\section{Minimal model program for very basic slc-trivial fibrations}\label{sec6}

At the beginning of this section, 
we should note that the proof of \cite{fl-plt}*{Section 7}
works for any plt pairs with $\kappa(K_X+\Delta) <\dim X$ once we assume
that Conjecture \ref{ab-conj} holds true in lower dimensions 
(playing the same role as \cite{fl-plt}*{Corollary 5.5} in the proof of \cite{fl-plt}*{Theorem 1.2}),
and the termination of the minimal model program for very basic slc-trivial fibrations
(playing the same role as \cite{fl-plt}*{Section 6} in the proof of \cite{fl-plt}*{Theorem 1.2}).
Since we have Theorem \ref{main-thm}, we can prove Theorem \ref{main-cor} immediately
once we establish  the termination of the minimal model program for 
very basic slc-trivial fibrations when the base is of dimension 3.
Therefore, in the rest of this section, we try to explain the minimal model program 
for very basic slc-trivial fibrations in some general settings.

\begin{defn}\label{d-flip}
Let $f\colon (X, B_X) \to (Y,D)$ be a very basic slc-trivial fibration
with the structure decomposition $D=K_Y+B_Y+M_Y$.
The commutative diagram
$$
\xymatrix{
Y \ar[rd]_{g}\ar@{-->}[rr]^p & & Y^+ \ar[ld]^{g^+}\\ 
&Z&
}
$$ 
is called a {\em{flip with respect to $D$}} or 
{\em{$D$-flip}} for short, if
\begin{itemize}
\item[(1)] $D^+$ is $\mathbb Q$-Cartier, where $D^+$ is the strict transform of $D$ on $Y^+$;
\item[(2)] $\Ex(g)$ has codimension at least two in $Y$ and $-D$ is $g$-ample;
\item[(3)]  $\Ex(g^+)$ has codimension at least two in $Y^+$ and $D^+$ is $g^+$-ample.
\end{itemize} 
\end{defn}

We show that $D$-flips perform in the framework of very basic slc-trivial fibrations.

\begin{lem}\label{flip-lem}
Let $f\colon (X, B_X) \to (Y,D)$ be a very basic slc-trivial fibration
with the structure decomposition $D=K_Y+B_Y+M_Y$. Let
$$
\xymatrix{
Y \ar[rd]_{g}\ar@{-->}[rr]^p & & Y^+ \ar[ld]^{g^+}\\ 
&Z&
}
$$ 
be a $D$-flip.
Then there is an induced very basic slc-trivial fibration $f^+\colon (X^+, B_{X^+}) \to (Y^+, D^+)$.
\end{lem}

\begin{proof}
Take a common resolution of $p$ such that
$$
\xymatrix{
&Y'\ar[ld]_{\alpha}\ar[rd]^{\beta}\\ 
Y \ar@{-->}[rr]^p & & Y^+
}
$$
commutes over $Z$.
Let $K_{Y'}=\alpha^*D+E_1$
and 
$K_{Y'}=\beta^*D^++E_2$.
Let $h=g\circ\alpha=g^+\circ \beta$. Then $E_1-E_2$ is $h$-nef 
and exceptional since $g_*(D)=g^+_*(D^+)$.
By the negative lemma (cf. \cite{kollar-mori}*{Lemma 3.39}), $E_2-E_1$ is effective.
Let $E_1+\alpha^*M_Y=E^+-E^-$ where $E^+$ and $E^-$ are both effective.
Then $K_{Y'}=\alpha^*D+E_1=\alpha^*(K_Y+B_Y)+E^+-E^-$.
By \cite{fujino-slc-trivial}*{Theorem 1.7}, $E^-$ is a boundary.
 Therefore, we can 
run the relative minimal model program with respect to 
$K_{Y'}+E^-+\alpha^*M_Y\sim_{\mathbb Q, \alpha}E^+$ over $Y$
and contract the divisor $E^+$ exactly.  
Replacing $Y'$ by the relative minimal model,
we can assume that  $B_{Y'}$ (the pushforward of $E^-$) is effective.
Note that 
$$
K_{Y'}+E^-+\alpha^*M_Y\sim_{\mathbb Q, \beta}E^++(E_2-E_1)
$$ over $Y^+$.
Therefore, by viewing the contraction of $E^+$ as steps of the relative minimal model program over $Y^+$,
we can see that $\beta$ is still a morphism since $E_2-E_1$ is effective. 
Then there is an induced very basic slc-trivial fibration on $Y'$ with 
$D'=\alpha^*D=\beta^*D^++(E_2-E_1)$ by Lemma \ref{res-lem}
and  an  induced very basic slc-trivial fibration 
on $Y^+$ with $D^+$ by Lemma \ref{slc-lem}.
\end{proof}

As we mentioned before, the pair $(Y, B_Y+M_Y)$ is a generalized lc pair in the sense of 
\cite{bz}*{Definition 4.1}. With a scaling, we can run the LMMP on $K_Y+B_Y+M_Y$ as in \cite{bz}*{Section 4}.
In general, it is hard to answer the termination of the LMMP unless adding suitable assumptions,
cf. \cite{bz}*{Lemma 4.4}. But in dimension 3, the LMMP of generalized lc pairs is settled down,
combining the fully understanding of the MMP of lc pairs. We refer to \cite{hl} for more details.
A private note of Han preparing for \cite{hl} shows the details of the termination of the LMMP in dimension 3.
Following Lemma \ref{flip-lem} and the understanding of the LMMP of generalized lc pairs in dimension 3,
we have the following results.
 
\begin{lem}\label{term-lem}
$D$-flips terminate in dimension three.
\end{lem}

\begin{lem}\label{mmp-lem}
We can run the $D$-minimal model program 
 beginning from a very basic slc-trivial fibration $f\colon (X, B_X) \to (Y,D)$ 
where $Y$ is a  $\mathbb Q$-factorial threefold and $D$ is big, 
ending up with a very basic slc-trivial fibration $f^*\colon (X^*, B_{X^*}) \to (Y^*, D^*)$ such that 
$D^*$ is nef and big. Furthermore, if $(Y, B_Y)$ is plt (resp. dlt), then so is $(Y^*, B_{Y^*})$.
\end{lem}

Now we are ready to prove Theorem \ref{main-cor}:

\begin{proof}[Sketch of proof]
We reduce to prove that the ring $R(Y, D)$ is finitely generated
for a very basic lc-trivial fibration $f\colon (X, B_X) \to (Y,D)$ where $Y$ is
a  $\mathbb Q$-factorial threefold, $(Y, B_Y)$ is plt and $D$ is big.
By Lemma \ref{mmp-lem}, we can further assume that $D$ is nef.
Then it suffices to prove that $D$ is semi-ample.
By Kawamata--Shokurov basepoint-free theorem (cf. \cite{fl-plt}*{Lemma 4.3}), 
it suffices to prove that $D|_{B^{=1}_Y}$ is semi-ample where $B^{=1}_Y$ is normal.
By connectedness lemma (see \cite{fl-plt}*{Corollary 4.2} or Corollary \ref{ind-cor}),
there is an induced very basic klt-trivial fibration on $B^{=1}_Y$.
It follows from Theorem \ref{main-thm} that $D|_{B^{=1}_Y}$ is semi-ample. 
This is what we want.
\end{proof}



\end{document}